\documentclass[pdflatex,sn-mathphys-num]{sn-jnl}

\usepackage{graphicx}%
\usepackage{multirow}%
\usepackage{amsmath,amssymb,amsfonts}%
\usepackage{amsthm}%
\usepackage{mathrsfs}%
\usepackage[title]{appendix}%
\usepackage{xcolor}%
\usepackage{textcomp}%
\usepackage{manyfoot}%
\usepackage{booktabs}%
\usepackage{algorithm}%
\usepackage{algorithmicx}%
\usepackage{algpseudocode}%
\usepackage{listings}%

\theoremstyle{thmstyleone}%
\newtheorem{theorem}{Theorem}[section]
\newtheorem{lemma}[theorem]{Lemma}

\newtheorem{corollary}[theorem]{Corollary}

\usepackage{url}
\theoremstyle{thmstyletwo}%

\theoremstyle{thmstylethree}%
\newtheorem{definition}{Definition}%

\begin{document}

\title[RKHS on Banach Completions of VPD Groups]{Reproducing Kernel Hilbert Spaces on Banach Completions of Virtual Persistence Diagram Groups}


\author*{\fnm{Charles} \sur{Fanning}}\email{cfannin8@students.kennesaw.edu}

\author{\fnm{Mehmet} \sur{Aktas}}\email{maktas1@kennesaw.edu}

\affil*{\orgdiv{School of Data Science and Analytics}, 
\orgname{Kennesaw State University}, 
\orgaddress{\street{1000 Chastain Rd NW}, 
\city{Kennesaw}, 
\postcode{30144}, 
\state{Georgia}, 
\country{United States}}}


\abstract{Persistent homology maps a simplicial complex filtered by elements in $\mathbb R$ to finite formal sums of elements of $\mathbb R_{\leq}^{2} = \{ (b,d) \in \mathbb R^2 \cup \{ \infty \} \mid b < d \}$ called (finite) persistence diagrams. This map is stable with respect to the $p$--Wasserstein distance for all $p \in \left[1, + \infty \right]$. Bubenik and Elchesen extend the free translation-invariant commutative Lipschitz monoid of finite persistence diagrams $D(X,A) = D(X)/D(A)$ on arbitrary metric pairs $(X,d,A)$ with $A \subset X$ onto the free translation-invariant abelian Lipschitz group of virtual persistence diagrams $K(X,A) = K(X)/K(A)$ as an isometric embedding $D(X,A) \hookrightarrow K(X,A)$ via the Grothendieck group completion. They prove that the $p$-Wasserstein distance is translation invariant on $D(X,A)$ if and only if $p=1$ and define the unique translation-invariant embedding of $W_1[d]$ into $K(X,A)$ as $\rho.$ When $K(X,A)$ is locally compact abelian, translation-invariant kernels can be constructed via positive-definite functions and Bochner's theorem on the Pontryagin dual. We prove that, for the metric topology induced by $\rho$, the group $(K(X,A),\rho)$ is locally compact if and only if it is discrete, equivalently when the pointed metric space $(X/A,d_1,[A])$ is uniformly discrete, and hence this approach fails outside that case. Assuming instead that $(X/A,d_1,[A])$ is separable and not uniformly discrete, we develop a translation-invariant kernel theory for non--locally compact virtual persistence diagram groups. The group $K(X,A)$ embeds isometrically into its canonical Banach-space linearization $B=\widehat V(X,A)\cong\mathcal F(X/A,d_1)$, and each bounded symmetric positive operator $Q\colon B\to B^\ast$ determines a translation-invariant Gaussian kernel $k(x,y)=\exp\!\left(-\tfrac12\,\langle Q(x-y),x-y\rangle_{B,B^\ast}\right).$}




\maketitle

\section{Introduction}

Persistent homology assigns to a filtered simplicial complex a persistence diagram: a finite multiset of points in $\mathbb R^2$ encoding the birth and death parameters of homological features. Under the usual tameness assumptions ensuring finitely many homological critical values, the resulting diagram is Lipschitz-stable on the input. More precisely, the map from filtrations (or functions) equipped with the $\|\cdot\|_\infty$ metric to persistence diagrams equipped with the bottleneck distance is $1$--Lipschitz, and, under appropriate summability conditions on lifetimes, is also Lipschitz with respect to $p$--Wasserstein distances \cite{892133,CohenSteiner2007StabilityPD,Oudot2015PersistenceT}.

Kernel and embedding methods for persistence diagrams are by now well developed. Prominent examples include \emph{persistence images}, which map diagrams to finite-dimensional $L^2$-type representations \cite{adams2017persistenceimages}, the \emph{persistence scale-space kernel} \cite{reininghaus2014stablemultiscale}, the \emph{persistence weighted Gaussian kernel} \cite{kusano2016pwgk}, the \emph{Sliced Wasserstein kernel} \cite{carriere2017slicedwasserstein}, and the \emph{persistence Fisher kernel} \cite{le2018persistencefisher}. However, they do not operate on the intrinsic algebraic structure of the space of persistence diagrams itself. That space is a free commutative, cancellative monoid generated by the birth--death space, with monoid operation given by addition of multiplicities and equipped with bottleneck or Wasserstein matching metrics. Since ordinary persistence diagrams do not form an LCA group, these constructions do not have access to the canonical functorial identification as in the case of Pontryagin duality and Bochner's theorem. 

Virtual persistence diagrams, introduced by Bubenik and Elchesen~\cite{bubenik2022virtualpd}, are obtained by adjoining additive inverses to finite persistence diagrams in a manner compatible with matching metrics. Given a metric pair $(X,d,A)$, let $D(X,A)$ denote the commutative monoid of finite diagrams relative to $A$, and let $K(X,A)$ denote its Grothendieck group. For the Wasserstein--$1$ distance, \cite{bubenik2022virtualpd} defines a translation-invariant metric $\rho$ on $K(X,A)$ extending the Wasserstein--$1$ metric on $D(X,A)$. More generally, for $p\in[1,\infty]$, the $p$-Wasserstein distance on $D(X,A)$ is translation invariant if and only if the associated strengthened metric $d_p$ is a $p$-metric~\cite{bubenik2022virtualpd}. In particular, $p=1$ is the unique exponent for which translation invariance holds for every metric pair. Since translation invariance is required to extend a matching metric from diagrams to the Grothendieck group, this identifies the Wasserstein--$1$ metric as the canonical choice for virtual persistence diagram groups.

Beyond the group level, the Grothendieck construction admits a canonical linearization. The group $(K(X,A),\rho)$ embeds isometrically into a normed vector space $(V(X,A),\|\cdot\|_{W_1})$, and its Banach completion $B:=\widehat V(X,A)$ provides a canonical Banach-space model of the $W_1$ geometry of virtual persistence diagrams \cite{bubenik2022virtualpd}. Under the standing assumptions fixed in Section~\ref{sec:background} (in particular, that the quotient by $A$ equipped with the $1$-strengthened metric yields a separated pointed metric space), one has a canonical isometric isomorphism
\[
B \;\cong\; \mathcal F(X/A,d_1),
\]
the Lipschitz-free (Arens--Eells) Banach space over the pointed metric space $(X/A,d_1,[A])$ \cite{bubenik2022virtualpd,weaver2018lipschitz}. 

In the uniformly discrete case, the metric group $(K(X,A),\rho)$ is discrete and locally compact, and the duality between $K(X,A)$ and its Pontryagin dual supports a correspondence in which analytic constructions on $K(X,A)$ can be recovered from their dual representations, as recalled in Section~\ref{subsec:background-lc-rkhs} \cite{Folland2015CourseAH,fanning2025reproducingkernelhilbertspaces}. In Section~\ref{sec:classification}, we show that this is the only situation in which such a correspondence is available for the metric topology induced by $\rho$: $(K(X,A),\rho)$ is locally compact if and only if it is discrete, and this holds exactly when the pointed metric space $(X/A,d_1,[A])$ is uniformly discrete. The argument combines general properties of the Grothendieck metric with rigidity results for locally compact group topologies on free abelian groups, as stated, for example, in~\cite{morris1977locallycompact}.

Outside the uniformly discrete case, the Grothendieck metric group of virtual persistence diagrams fails to be locally compact with respect to the metric topology induced by $\rho$, and no translation-invariant measure or compatible dual group structure supports the construction of reproducing kernel Hilbert spaces. This paper proceeds through the canonical Banach-space linearization of virtual persistence diagrams and defines translation-invariant Gaussian kernels directly on that space. These kernels depend on continuous Hilbertian seminorms on the Banach model, equivalently on bounded symmetric positive operators $Q\colon B\to B^\ast$, and admit the explicit form $k(x,y)=\exp\!\left(-\tfrac12\,s(x-y)^2\right), $ where $s(x)^2=\langle Qx,x\rangle_{B,B^\ast}.$

\subsection{Our contributions}

Throughout, we work under the standing hypotheses~\emph{(H1)}--\emph{(H2)} in Definition~\ref{def:standing-hypotheses}, so that $(K(X,A),\rho)$ is non-discrete and the Banach completion $B=\widehat V(X,A)\cong\mathcal F(X/A,d_1)$ is separable and infinite-dimensional (Corollary~\ref{cor:input-output-discrete}, Lemmas~\ref{lem:infinite-rank} and~\ref{lem:B-separable}).

The first contribution is a family of \emph{explicit metric bounds} for reproducing kernel Hilbert spaces (RKHS) associated with translation-invariant Gaussian kernels on $B$. Given admissible covariance data $(J,\Sigma,t)$ as in Section~\ref{subsec:main-rkhs}, Theorem~\ref{thm:lipschitz-main-sigma} controls the $\rho$--Lipschitz constants of all functions in the RKHS $H_{J,\Sigma,t}$ restricted to $K(X,A)$ by a single scalar factor derived from $\sum_{n\ge1}\sigma_n w_n^2$. Theorem~\ref{thm:entropy-feature} bounds covering numbers of subsets of $K(X,A)$ with respect to the feature metric $\delta_{J,\Sigma,t}$ by covering numbers with respect to $\rho$. Together, these results quantify how the covariance operator $Q_{J,\Sigma}=J^\ast\Sigma J\colon B\to B^\ast$ restricts the regularity and metric complexity of kernel-based constructions on virtual persistence diagrams.

The second contribution uses the Banach model to extract information in the reverse direction. Although translation-invariant Gaussian kernels on $B$ are noninjective on $K(X,A)$, Theorem~\ref{thm:finite-max-certificate-slack} shows that a single kernel value $k_{J,\Sigma,t}(g,0)$, together with finite support information for $g$, provides an explicit upper bound on the diagrammatic mass $\mathcal M(g)$, a metric invariant defined directly from $(X/A,d_1)$. Theorem~\ref{thm:seminorm-bilipschitz-iff} then compares different choices of covariance operators: it gives an exact Rayleigh-quotient criterion for when two operators $Q_1,Q_2\colon B\to B^\ast$ induce Hilbertian seminorms that are bi-Lipschitz equivalent on $B$, and hence when the corresponding Gaussian kernels induce uniformly equivalent feature geometries on $K(X,A)$.

The third contribution concerns computable approximations. Section~\ref{subsec:main-rff} implements random Fourier feature maps $\Phi_R\colon B\to\mathbb C^R$ based on the Gaussian measure $\gamma_\Sigma$ on $\ell^2$. Lemma~\ref{lem:rff-uniform} gives a uniform concentration bound for the empirical kernel $\widehat k_R$ on finite subsets of $B$, while Corollary~\ref{cor:rff-lipschitz} controls the Lipschitz constant of $\Phi_R$ with respect to $\|\cdot\|_{W_1}$. Corollary~\ref{cor:rff-mass} transfers the mass bound from Theorem~\ref{thm:finite-max-certificate-slack} to the empirical kernel $\widehat k_R$, and Corollary~\ref{cor:rff-entropy} transfers the covering number comparison from $K(X,A)$ to its image in the finite-dimensional feature space. These results convert the Banach--Gaussian construction into a finite-sample scheme with explicit probabilistic error control.

Finally, Section~5 applies this analytic pipeline to virtual persistence diagrams obtained from lower--star clique filtrations indexed by partially ordered metric label spaces. Edge labelings into $\mathbb R$, $\mathbb R^3$, $C([0,1])$, and $\mathcal S^+_3$ produce birth--death data in product spaces of the form $(P^2,d_P\oplus d_P,A)$ and hence virtual diagrams in a common Grothendieck group $(K(X,A),\rho)$ with Banach completion $B=\widehat V(X,A)$ (Figures~\ref{fig:rff-network-labels} and~\ref{fig:rff-network-vpd}). The bounds from Section~\ref{sec:main} then isolate the contribution of the label-space geometry to Lipschitz constants, random-feature concentration, and robustness of kernel evaluations under perturbations of persistence data.\footnote{\url{https://github.com/cfanning8/Reproducing_Kernel_Hilbert_Spaces_for_Non_Discrete_Virtual_Persistence_Diagrams}}

\subsection{Related work}

Foundational results on persistent homology, barcodes/persistence diagrams, and stability were established in \cite{892133,Zomorodian2005ComputingPH,CohenSteiner2007StabilityPD,Oudot2015PersistenceT}, with subsequent structural developments including stability and structure theorems for persistence modules \cite{chazal2013structurestabilitypersistencemodules}. Extensions such as extended persistence \cite{CohenSteiner2009ExtendingPU}, zigzag persistence \cite{carlsson2008zigzagpersistence}, and vineyards \cite{CohenSteiner2006VinesAV} broaden the classical one-parameter setting, while multiparameter persistence replaces complete discrete invariants by interleaving-based analysis \cite{Lesnick2015InterleavingMP,chazal2013structurestabilitypersistencemodules}. Virtual persistence diagrams and their Grothendieck/Banach-space formalism are developed in \cite{bubenik2022virtualpd}. RKHS constructions based on multipliers on the Pontryagin dual in the uniformly discrete (locally compact) regime are developed in \cite{fanning2025reproducingkernelhilbertspaces}.

\subsection{Organization of the paper}

\begin{itemize}
\item Section~\ref{sec:background} introduces virtual persistence diagrams, the $1$--Wasserstein Grothendieck metric, and the Banach-space model.

\item Section~\ref{sec:classification} characterizes local compactness of $(K(X,A),\rho)$ and its equivalence with uniform discreteness of $(X/A,d_1,[A])$.

\item Section~\ref{sec:main} constructs translation-invariant Gaussian kernels on the Banach completion, together with Lipschitz, mass, entropy, and random feature bounds.

\item Section~\ref{sec:example} demonstrates the full pipeline on graph-based filtrations with non--uniformly discrete label spaces.
\end{itemize}

\section{Background and Notation}
\label{sec:background}

Throughout this paper, $(X,d,A)$ denotes a metric pair, consisting of a metric space $(X,d)$ together with a distinguished subset $A \subseteq X$. We factor through the associated pointed metric space obtained by collapsing $A$ to a single point, write $X/A$ for the resulting quotient space equipped with the induced metric and basepoint $[A]$, and freely pass between the pair $(X,d,A)$ and the pointed space $(X/A,d)$.

\subsection{Virtual persistence diagrams}
\label{subsec:background-vpd}

Let $(X,d,A)$ be a metric pair. For each $x \in X$, define \begin{equation}\label{eq:dXA} d(x,A) := \inf_{a \in A} d(x,a). \end{equation} The \emph{$1$-strengthened metric} on $X$ is defined by \begin{equation}\label{eq:d1}
d_1(x,y) := \min\bigl(d(x,y),\, d(x,A) + d(y,A)\bigr), \qquad x,y \in X.
\end{equation}
Let $X/A$ denote the quotient space obtained by collapsing $A$ to a single point, denoted $[A]$. The function $d_1$ descends to a metric on $X/A$, which we again denote by $d_1$.

\begin{definition}\label{def:finite-diagram}
Let $D(X)$ denote the free commutative monoid on $X$. The quotient
\[
D(X,A) := D(X)/D(A)
\]
is called the space of \emph{finite persistence diagrams} associated to the metric pair $(X,d,A)$. The zero element is denoted by $0$.
\end{definition}

For $\alpha,\beta \in D(X,A)$, a matching between $\alpha$ and $\beta$ is an element of $D(X \times X)$ whose marginals agree with $\alpha$ and $\beta$ modulo $D(A)$. The $1$-Wasserstein distance is defined by
\begin{equation}\label{eq:W1}
W_1(\alpha,\beta) := \inf_{\sigma} \sum_{(x,y)\in\sigma} d_1(x,y),
\end{equation}
where the infimum is taken over all such matchings.

The choice of $p=1$ is not arbitrary. For $p \in [1,\infty]$, define the $p$-strengthened metric
\begin{equation}\label{eq:dp}
d_p(x,y) := \min\bigl(d(x,y),\, (d(x,A),d(y,A))_p\bigr).
\end{equation}
A metric $d$ is called a $p$-metric if
\[
d(x,y) \le (d(x,z),d(z,y))_p \quad\text{for all } x,y,z \in X.
\]

\begin{theorem}\label{thm:p-canonical} Let $(X,d,A)$ be a metric pair and let $p \in [1,\infty]$. The $p$-Wasserstein distance $W_p[d]$ on $D(X,A)$ is translation invariant if and only if the strengthened metric $d_p$ is a $p$-metric. \end{theorem}

This characterization is due to \cite{bubenik2022virtualpd}. In particular, $W_1[d]$ is translation invariant for every metric pair, whereas for metric spaces with nontrivial length geometry, translation invariance fails for all $p>1$. Since translation invariance is necessary for extending the Wasserstein distance to the Grothendieck group of diagrams, the choice $p=1$ is forced and canonical.

Let $K(X,A)$ denote the Grothendieck group of the commutative monoid $D(X,A)$. Each element of $K(X,A)$ is represented by a formal difference $\alpha-\beta$ with $\alpha,\beta \in D(X,A)$. Define
\begin{equation}\label{eq:rho}
\rho(\alpha-\beta,\gamma-\delta) := W_1(\alpha+\delta,\gamma+\beta).
\end{equation}

\begin{theorem}\label{thm:rho}
The function $\rho$ is a well-defined translation-invariant metric on $K(X,A)$~\cite{bubenik2022virtualpd} and extends $W_1$ along the canonical embedding $D(X,A)\hookrightarrow K(X,A)$.
\end{theorem}

Let $V(X)$ denote the free real vector space on $X$ and define $V(X,A):=V(X)/V(A)$. \cite{bubenik2022virtualpd} defines a canonical norm $\|\cdot\|_{W_1}$ on $V(X,A)$ extending $\rho$ on $K(X,A)$. Let $\widehat V(X,A)$ denote the Banach completion.

\begin{theorem}\label{thm:BE-chain}
There is a chain of isometric embeddings~\cite{bubenik2022virtualpd}
\[
(D(X,A),W_1)\hookrightarrow (K(X,A),\rho) \hookrightarrow (V(X,A),\|\cdot\|_{W_1}) \hookrightarrow (\widehat V(X,A),\|\cdot\|_{W_1}).
\]
\end{theorem}

\begin{figure}[t]
 \centering
 \includegraphics[width=0.75\textwidth]{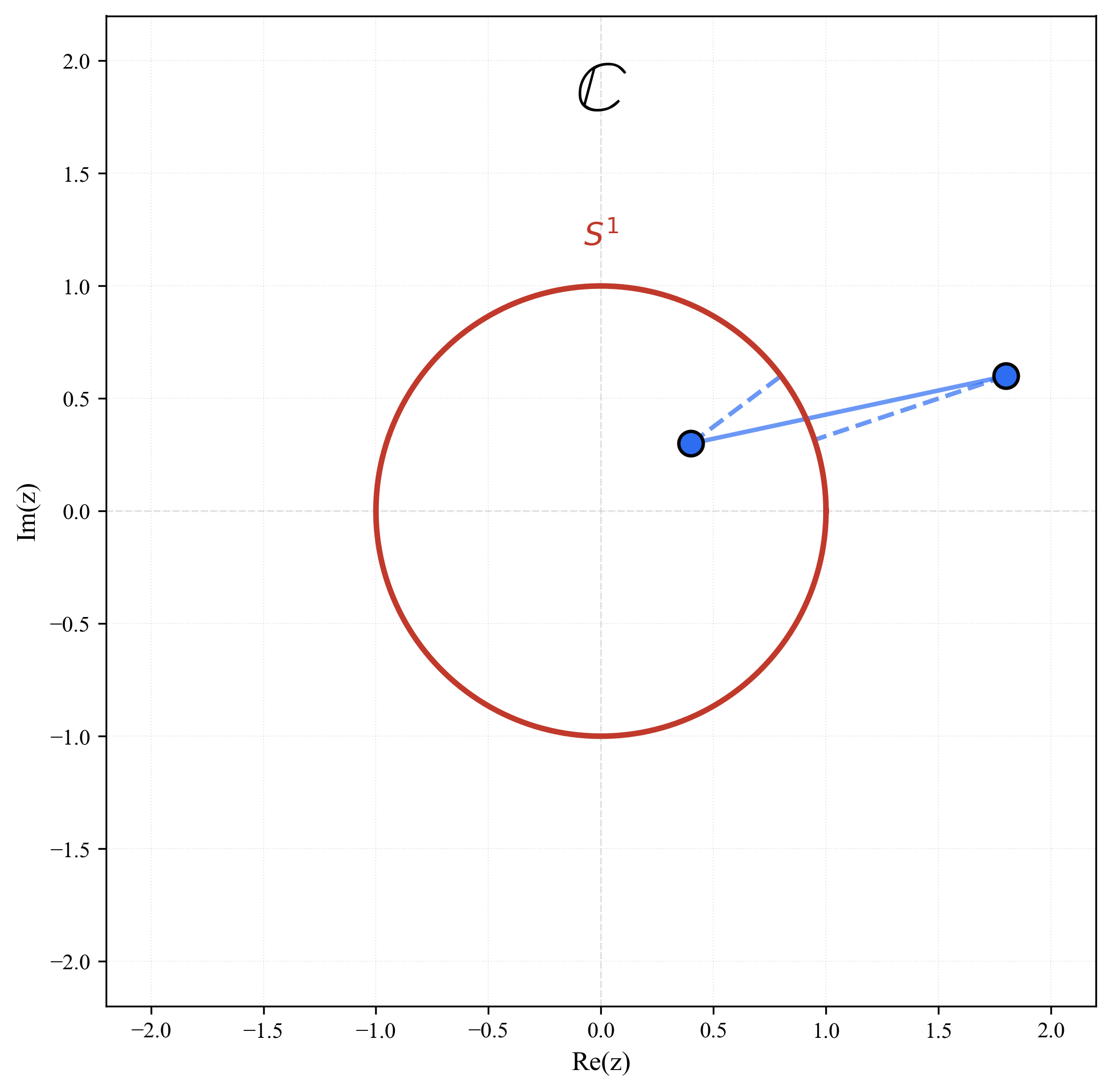}
 \caption{The metric pair \((X,d,A)=(\mathbb C,|\cdot|,S^1)\), where \(A=S^1\) is the diagonal. The strengthened metric is \(d_1(x,y)=\min\bigl(|x-y|,\; \bigl||x|-1\bigr|+\bigl||y|-1\bigr|\bigr)\). The geometry shown uses the Euclidean norm for visualization; throughout the paper, the matching metric is defined using the \(\ell_1\) norm.}
 \label{fig:metric-pair-complex}
\end{figure}

\begin{figure}[t]
 \centering
 \includegraphics[width=0.7\textwidth]{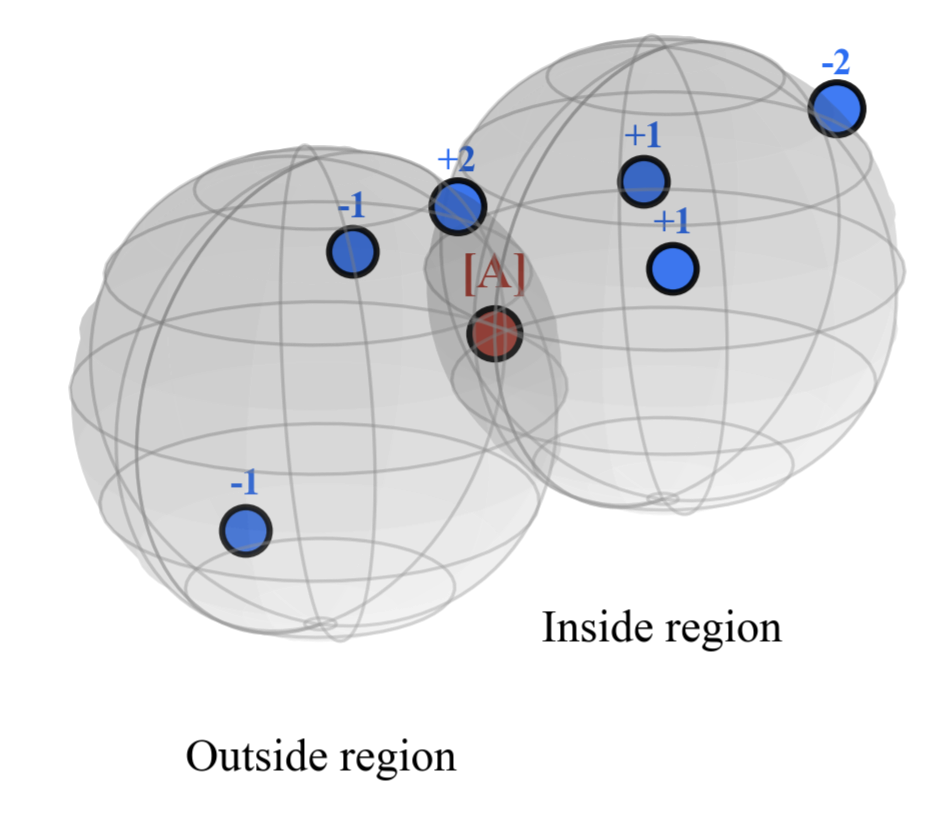}
 \caption{Continuing from Figure~\ref{fig:metric-pair-complex}, the quotient \(X/A = \mathbb C/S^1 \) obtained by collapsing the diagonal \(S^1\subset\mathbb C\) to a point \([A]\). The resulting space is homeomorphic to the wedge \(S^2\vee S^2\), corresponding to points with \(|z|<1\) and \(|z|>1\). A virtual persistence diagram is a finite signed multiset on \(X/A\setminus\{[A]\}\).}
 \label{fig:quotient-vpd}
\end{figure}

\subsection{Reproducing Kernel Hilbert Spaces for Finite Virtual Persistence Diagrams}
\label{subsec:background-lc-rkhs}

Assume that $(X,d,A)$ is finite. Then the group of virtual persistence diagrams
admits a canonical isomorphism
\[
K(X,A)\cong \mathbb{Z}^{|X\setminus A|}
\]
as a discrete locally compact abelian group, whose Pontryagin dual is the torus
$\mathbb{T}^{|X\setminus A|}$ equipped with normalized Haar measure. Characters are
given by
\[
\chi_\theta(\gamma) := e^{i\langle \gamma,\theta\rangle},
\qquad
\theta\in\mathbb{T}^{|X\setminus A|},\ \gamma\in K(X,A).
\]

Following \cite{fanning2025reproducingkernelhilbertspaces}, the quotient metric space $(X/A,d_1)$ is modeled by a finite connected weighted graph whose shortest-path metric agrees with $d_1$. The associated graph Laplacian induces a nonnegative Dirichlet symbol
\[
\lambda:\mathbb{T}^{|X\setminus A|}\to[0,\infty),
\]
which quantifies the oscillation of characters with respect to the Wasserstein-lifted metric on $K(X,A)$.

For each $t>0$, define a finite positive measure on $\mathbb{T}^{|X\setminus A|}$ by
\[
d\nu_t(\theta) := e^{-t\lambda(\theta)}\,d\mu(\theta),
\]
where $\mu$ denotes Haar measure. The corresponding Fourier--Stieltjes transform defines a translation-invariant positive definite kernel
\[
k_t(\gamma,\eta)
:= \int_{\mathbb{T}^{|X\setminus A|}} \chi_\theta(\gamma-\eta)\,d\nu_t(\theta),
\qquad
\gamma,\eta\in K(X,A),
\]
and hence a reproducing kernel Hilbert space $H_t$ of functions on $K(X,A)$.

The main analytic results of \cite{fanning2025reproducingkernelhilbertspaces} provide explicit global Lipschitz bounds for functions in $H_t$ with respect to the Grothendieck metric $\rho$. In spectral form, one obtains the following estimate.

\begin{corollary}[Spectral form {\cite[Corollary~2]{fanning2025reproducingkernelhilbertspaces}}]
For every $t>0$ and $f\in H_t$,
\[
\mathrm{Lip}_\rho(f)
\le
\frac{\pi}{d_{\min}\sqrt{2w_{\min}}}\,
\|f\|_{H_t}
\left(
\int_{\mathbb{T}^{|X\setminus A|}}
\lambda(\theta)\,e^{-t\lambda(\theta)}\,d\mu(\theta)
\right)^{1/2},
\]
where $w_{\min}$ and $d_{\min}$ are the minimal edge weight and minimal edge length, respectively.
\end{corollary}

This bound admits a purely geometric reformulation in terms of phase functions on $(X/A,d_1)$.

\begin{corollary}[Geometric form {\cite[Corollary~3]{fanning2025reproducingkernelhilbertspaces}}]
For every $t>0$ and $f\in H_t$,
\begin{align*}
\mathrm{Lip}_\rho(f)
\le {} & \|f\|_{H_t}
\left(
\int_{\mathbb{T}^{|X\setminus A|}}
\mathrm{Lip}_{d_1}(\phi_\theta)^2
\right. \\
& \left.
\times
\exp\!\left(
-\,t\,\frac{8w_{\min}d_{\min}^2}{\pi^4}\,
\mathrm{Lip}_{d_1}(\phi_\theta)^2
\right)
\,d\mu(\theta)
\right)^{1/2}.
\end{align*}
Here $\phi_\theta$ denotes the phase function associated with $\chi_\theta$.
\end{corollary}

\section{Classification of Virtual Persistence Diagrams}
\label{sec:classification}

In this section, we classify the metric abelian group of virtual persistence diagrams
\((K(X,A),\rho)\) with respect to local compactness in the topology induced by the
Grothendieck metric \(\rho\). We show that \((K(X,A),\rho)\) is locally compact if and only if
the metric topology is discrete.

\begin{theorem}\label{thm:morris}
If $F$ is a free abelian group that is locally compact Hausdorff, then $F$ has the
discrete topology.
\end{theorem}

This appears as Theorem~31 in~\cite{morris1977locallycompact}.

\begin{theorem}\label{thm:classification}
Let $(X,d,A)$ be a metric pair, and let $\rho$ denote the Grothendieck metric on $K(X,A)$. With respect to the metric topology induced by $\rho$, the topological group $(K(X,A),\rho)$ is locally compact if and only if it is discrete.
\end{theorem}

\begin{proof}
Since $\rho$ is a translation-invariant metric on the abelian group $K(X,A)$, it induces a Hausdorff group topology. Translation invariance means that for all $g,g',h\in K(X,A)$, $\rho(g+h,g'+h)=\rho(g,g').$ In particular, translations are isometries. Inversion is also an isometry: for all $g,g'\in K(X,A)$, translation invariance gives $\rho(-g,-g')=\rho(g,g').$

To verify joint continuity of addition, fix $g,g',h,h'\in K(X,A)$ and define $u:=g-g'$ and $v:=h-h'$. Using translation invariance and the triangle inequality, we obtain
\begin{align}\label{eq:rho-addition-continuity}
\rho(g+h,g'+h')
&=\rho\bigl((g+h)-(g'+h'),0\bigr)\notag\\
&=\rho(u+v,0)\notag\\
&=\rho(u,-v)\notag\\
&\le \rho(u,0)+\rho(0,-v)\notag\\
&=\rho(u,0)+\rho(v,0)\notag\\
&=\rho(g,g')+\rho(h,h').
\end{align}
Thus, addition is jointly continuous, and $(K(X,A),\rho)$ is a Hausdorff topological abelian group.

We now identify $K(X,A)$ algebraically. The commutative monoid $D(X)$ is free on $X$. Passing to the quotient $D(X,A)=D(X)/D(A)$ forces every generator supported in $A$ to be identified with $0$. Equivalently, $D(X,A)$ satisfies the universal property of the free commutative monoid on $X\setminus A$. Its Grothendieck completion is therefore the free abelian group on $X\setminus A$, that is, $K(X,A)\cong \bigoplus_{x\in X\setminus A}\mathbb Z.$

If $X\setminus A=\varnothing$, then $K(X,A)=\{0\}$, which is discrete and locally compact, and the conclusion follows immediately. Assume for the remainder of this proof that $X\setminus A\neq\varnothing$. 

If $(K(X,A),\rho)$ is discrete, then $\{0\}$ is an open neighborhood of the identity. Since $\{0\}$ is finite, it is compact, and therefore $(K(X,A),\rho)$ is locally compact.

Conversely, suppose $(K(X,A),\rho)$ is locally compact. Then it is a Hausdorff locally compact topological group whose underlying group is free abelian. By Theorem~\ref{thm:morris}, any locally compact Hausdorff group topology on a free abelian group is necessarily discrete. Hence the topology induced by $\rho$ on $K(X,A)$ is discrete.
\end{proof}

Theorem~\ref{thm:classification} reduces the question of local compactness for
$(K(X,A),\rho)$ to discreteness of the metric topology induced by $\rho$. Since the
Grothendieck metric is defined functorially from the $1$--strengthened metric on the
pointed space $(X/A,d_1,[A])$, the following corollary identifies the corresponding
metric condition on $(X/A,d_1,[A])$. Corollary~\ref{cor:input-output-discrete} shows that
discreteness of $(K(X,A),\rho)$ is equivalent to uniform discreteness of
$(X/A,d_1,[A])$.

\begin{corollary}\label{cor:input-output-discrete}
The following are equivalent:
\begin{enumerate}
\item[(i)] $(K(X,A),\rho)$ is locally compact.
\item[(ii)] $(K(X,A),\rho)$ is discrete.
\item[(iii)] $(X/A,d_1,[A])$ is uniformly discrete.
\end{enumerate}
\end{corollary}

\begin{proof}
The equivalence (i)$\iff$(ii) is Theorem~\ref{thm:classification}. We prove (ii)$\iff$(iii).

For $u\in X/A\setminus\{[A]\}$, let $e_u\in D(X,A)$ denote the singleton diagram supported at $u$. Fix $x\in X\setminus A$ with $\pi(x)=u$. Since the quotient map restricts to a bijection $\pi\colon X\setminus A\to X/A\setminus\{[A]\}$, the value $d(x,A)$ depends only on $u$, and we define $d_1(u,[A]):=d(x,A)$. For any $z\in X$ and any $a\in A$, the definition of the strengthened metric gives $d_1(z,a)=d(z,A)$, and therefore $d_1(\pi(z),[A])\le d_1(z,a)$.

We record two singleton identities. For distinct $u,v\in X/A\setminus\{[A]\}$, $W_1(e_u,e_v)=d_1(u,v)$, and for $u\in X/A\setminus\{[A]\}$, $W_1(e_u,0)=d_1(u,[A])$.

To prove the first identity, fix $x,y\in X\setminus A$ with $\pi(x)=u$ and $\pi(y)=v$. For the upper bound, matching $x$ directly to $y$ gives cost $d(x,y)$. Given $\eta>0$, choose $a,a'\in A$ with $d(x,a)\le d(x,A)+\eta$ and $d(y,a')\le d(y,A)+\eta$. Matching $x$ to $a$ and $a'$ to $y$ gives cost at most $d(x,A)+d(y,A)+2\eta$. Taking the infimum over matchings and letting $\eta\downarrow 0$ yields $W_1(e_u,e_v)\le \min(d(x,y),d(x,A)+d(y,A)) = d_1(u,v)$.

For the lower bound, let $\sigma$ be any admissible matching between $e_u$ and $e_v$, and let $\bar\sigma$ be its pushforward along $\pi\times\pi$, which is an element of $D((X/A)\times(X/A))$. We write $e_u$ also for the singleton generator at $u$ in $D(X/A)$, and write $\delta_{[A]}$ for the singleton generator at $[A]$. Replacing any coordinate $a\in A$ by $[A]$ cannot increase cost, so the total cost of $\bar\sigma$ computed with $d_1$ on $X/A$ is at most the total cost of $\sigma$ computed with $d_1$ on $X$. The marginals of $\bar\sigma$ have the form $\mathrm{pr}_{1\#}\bar\sigma=e_u+m\,\delta_{[A]}$ and $\mathrm{pr}_{2\#}\bar\sigma=e_v+m\,\delta_{[A]}$ for some $m\in\mathbb N$. Hence every pair in $\bar\sigma$ lies in $\{u,[A]\}\times\{v,[A]\}$. If $(u,v)$ occurs, then the cost is at least $d_1(u,v)$. Otherwise, the marginal constraints force both $(u,[A])$ and $([A],v)$ to occur, and the total cost is at least $d_1(u,[A])+d_1([A],v)\ge d_1(u,v)$ by the triangle inequality for $d_1$. Thus every admissible $\sigma$ has cost at least $d_1(u,v)$, proving the first identity.

For the second identity, the upper bound follows by matching $x$ to any $a\in A$, which gives $W_1(e_u,0)\le d_1(x,a)=d_1(u,[A])$. For the lower bound, let $\sigma$ be an admissible matching between $e_u$ and $0$, and let $\bar\sigma$ be its pushforward. The first marginal of $\bar\sigma$ contains one copy of $u$, while the second marginal contains none, so $\bar\sigma$ must contain at least one pair $(u,[A])$. Hence the total cost of $\bar\sigma$ is at least $d_1(u,[A])$, and since pushforward does not increase cost we obtain $W_1(e_u,0)\ge d_1(u,[A])$.

We now prove (ii)$\Rightarrow$(iii). Suppose $(X/A,d_1,[A])$ is not uniformly discrete. Then there exist distinct $u_n,v_n\in X/A$ with $d_1(u_n,v_n)\to 0$. Let $I:=\{n : u_n\neq[A]\text{ and }v_n\neq[A]\}$. If $I$ is infinite, pass to a subsequence indexed by $I$ and define $g_n:=e_{u_n}-e_{v_n}$. Then $\rho(g_n,0)=W_1(e_{u_n},e_{v_n})=d_1(u_n,v_n)\to 0$. If $I$ is finite, then for all sufficiently large $n$ exactly one of $u_n$ or $v_n$ equals $[A]$. If $v_n=[A]$ infinitely often, pass to that subsequence and define $g_n:=e_{u_n}$. Otherwise, $u_n=[A]$ infinitely often; pass to that subsequence and define $g_n:=-e_{v_n}$. In both cases, invariance of $\rho$ under inversion and the singleton identity give $\rho(g_n,0)=d_1(u_n,v_n)\to 0$. Hence $0$ is not isolated and $(K(X,A),\rho)$ is not discrete.

Finally, assume (iii) holds with constant $\varepsilon>0$ and let $g\in K(X,A)\setminus\{0\}$. Write $g=\alpha-\beta$ with $\alpha,\beta\in D(X,A)$. Since $\alpha\neq\beta$, any admissible matching $\sigma$ between $\alpha$ and $\beta$ has a pushforward $\bar\sigma$ whose marginals in $D(X/A)$ have the form $\alpha+m\,\delta_{[A]}$ and $\beta+m\,\delta_{[A]}$. If every pair $(u,v)$ in $\bar\sigma$ satisfied $u=v$, then the marginals would coincide, and we would have $\alpha=\beta$, a contradiction. Hence $\bar\sigma$ contains a pair $(u,v)$ with $u\neq v$, and uniform discreteness gives $d_1(u,v)\ge\varepsilon$. Therefore every admissible $\sigma$ has total cost at least $\varepsilon$, so $\rho(g,0)=W_1(\alpha,\beta)\ge\varepsilon$. Thus $\{0\}$ is open and $(K(X,A),\rho)$ is discrete.
\end{proof}

\section{Main Results}
\label{sec:main}

By Corollary~\ref{cor:input-output-discrete}, whenever $(K(X,A),\rho)$ is non-discrete, the metric topology induced by $\rho$ is not locally compact. The results of this section therefore proceed by working with the Banach completion $B:=\widehat V(X,A)$ of $K(X,A)$ with respect to $\|\cdot\|_{W_1}$, which is canonically isometric to the Lipschitz-free (Arens--Eells) Banach space $\mathcal F(X/A,d_1)$ by~\cite[Section~7]{bubenik2022virtualpd}.

\subsection{Hypotheses}
\label{subsec:main-hypotheses}

All results in this section are proved under the following standing assumptions.

\begin{definition}\label{def:standing-hypotheses}
The \emph{standing hypotheses} are:
\begin{enumerate}
\item[(H1)] The pointed metric space $(X/A,d_1,[A])$ is separable.
\item[(H2)] The pointed metric space $(X/A,d_1,[A])$ is not uniformly discrete.
\end{enumerate}
\end{definition}

Assumption~\emph{(H2)} is equivalent, by Corollary~\ref{cor:input-output-discrete}, to $(K(X,A),\rho)$ being non-discrete and hence not locally compact in the metric topology induced by $\rho$. The character-based construction recalled in Section~\ref{subsec:background-lc-rkhs}, which passes through Pontryagin duality and Haar measure on the dual group, requires local compactness and therefore does not apply under \emph{(H2)}. We consequently base our constructions on the Banach completion $B$.

Assumption~\emph{(H1)} guarantees that $B\cong \mathcal F(X/A,d_1)$ is separable (Lemma~\ref{lem:B-separable}). This separability is used to select countable norming families in $B^\ast$, and hence, to define explicit embeddings of $B$ into $\ell^2$ by means of countably many linear functionals. The resulting Gaussian kernel constructions are therefore genuinely countable and well-defined; moreover, they depend only on the induced covariance operator on $B$ and not on the particular choice of norming data used to realize it.

\subsection{Structural consequences}
\label{subsec:main-structural}

In the finite case treated in Section~\ref{subsec:background-lc-rkhs}, the discrete LCA structure of $K(X,A)$ and its Pontryagin dual torus support character-based Fourier--Stieltjes kernels. We work under the standing hypotheses of Definition~\ref{def:standing-hypotheses} to study the non--locally compact case, where Pontryagin duality no longer applies; in this subsection we introduce the Lipschitz, mass, and covering-number invariants through which we derive Gaussian RKHS bounds in Section~\ref{subsec:main-rkhs}.

\begin{lemma}\label{lem:infinite-rank}
Under assumption~\emph{(H2)}, the set $X\setminus A$ is infinite. Consequently, the abelian group $K(X,A)$ has infinite rank and the Banach space $B$ is infinite-dimensional.
\end{lemma}

\begin{proof}
If $X\setminus A$ were finite, then $X/A$ would be a finite metric space and therefore uniformly discrete, contradicting~\emph{(H2)}. Since $K(X,A)$ is the free abelian group on the set $X\setminus A$, it follows that $K(X,A)$ has infinite rank, and its Banach completion $B$ is infinite-dimensional.
\end{proof}

\begin{lemma}\label{lem:B-separable}
Under assumption~\emph{(H1)}, the Banach space $B$ is separable.
\end{lemma}

\begin{proof}
By~\cite[Corollary~7.10]{bubenik2022virtualpd}, the space $B$ is canonically isometric to the Lipschitz-free Banach space $\mathcal F(X/A,d_1)$. Lipschitz-free Banach spaces over separable pointed metric spaces are separable; see \cite{weaver2018lipschitz}, Section~3.1. The claim follows directly from~\emph{(H1)}.
\end{proof}

\begin{definition}\label{def:Lip-rho}
For a function $f\colon K(X,A)\to\mathbb R$, define
\[
\mathrm{Lip}_\rho(f)
:=
\sup_{\alpha\neq\beta}
\frac{|f(\alpha)-f(\beta)|}{\rho(\alpha,\beta)}\in[0,\infty].
\]
\end{definition}

For $g\in K(X,A)$, write $f(\cdot+g)$ for the translate $h\mapsto f(h+g)$. Since $\rho$ is translation invariant, one has $\mathrm{Lip}_\rho\bigl(f(\cdot+g)\bigr)=\mathrm{Lip}_\rho(f),$ for $g\in K(X,A).$ Define the canonical map from the pointed metric space $(X/A,d_1,[A])$ to $K(X,A)$ by $\iota\colon X/A\longrightarrow K(X,A),$ for $\iota([A])=0, \iota(u)=e_u, u\neq[A].$ Then $\rho(\iota(u),\iota(v))=d_1(u,v)$ for all $u,v\in X/A$. In particular, for all $u,v\in X/A$, $|f(\iota(u))-f(\iota(v))| \le \mathrm{Lip}_\rho(f)\,d_1(u,v),$ and hence $\mathrm{Lip}_{d_1}(f\circ \iota)\le \mathrm{Lip}_\rho(f).$ Moreover, if $f(0)=0$, then $f\circ \iota$ vanishes at the basepoint $[A]$.

\begin{definition}\label{def:mass}
For $g\in K(X,A)$, write $g=\sum_{u\in X/A\setminus\{[A]\}} n_u\,e_u$ with $n_u\in\mathbb Z$ and only finitely many nonzero coefficients. The \emph{mass} of $g$ is defined by
\[
\mathcal M(g):=\sum_{u\in X/A\setminus\{[A]\}} |n_u|\,d_1(u,[A]).
\]
For $\alpha\in D(X,A)$, this agrees with $\mathcal M(\alpha)=\sum_{x\in\alpha} d_1(x,[A])$.
\end{definition}

For $\alpha\in D(X,A)$, one has $W_1(\alpha,0)=\mathcal M(\alpha)$, since any admissible matching with the zero diagram pairs each point to the diagonal class $[A]$ and attains this cost. More generally, given $g=\sum_{u} n_u e_u$, define
\[
g_\pm:=\sum_{u} \max\{\pm n_u,0\}\,e_u\in D(X,A)
\]
so that $g=g_+-g_-$. The triangle inequality for $W_1$ then yields $\rho(g,0)=W_1(g_+,g_-)\le \mathcal M(g_+)+\mathcal M(g_-)=\mathcal M(g)$. Moreover,
\[
\mathcal M(g)
=
\inf\{\mathcal M(\alpha)+\mathcal M(\beta): \alpha,\beta\in D(X,A),\ g=\alpha-\beta\},
\]
and the infimum is attained by the canonical decomposition $g=g_+-g_-$.

\begin{definition}\label{def:covering-number}
For a subset $S\subset K(X,A)$ and $\varepsilon>0$, let $N(S,\varepsilon)\in\mathbb N\cup\{\infty\}$ denote the smallest integer $m$ such that $S$ can be covered by $m$ open $\rho$-balls of radius $\varepsilon$.
\end{definition}

Covering numbers are translation invariant in the sense that $N(S+g,\varepsilon)=N(S,\varepsilon),$ for $ g\in K(X,A).$ On singleton diagrams, covering numbers reduce to those of the underlying pointed metric space $(X/A,d_1,[A])$ via the identity $\rho(e_u,e_v)=d_1(u,v)$.

\subsection{Gaussian kernels induced by Hilbert embeddings}
\label{subsec:main-rkhs}

Identify the Banach space $B=\widehat V(X,A)$ canonically with the Lipschitz-free Banach space $\mathcal F(X/A,d_1)$. Via the dual isomorphism $B^\ast\cong \mathrm{Lip}_0(X/A,d_1)$ and the Kantorovich--Rubinstein dual representation of the norm, the closed unit ball of $B^\ast$ is weak-$\ast$ compact and metrizable. Fix a sequence $(\ell_n)_{n\ge1}\subset B^\ast$ with $\|\ell_n\|_{B^\ast}\le 1$ that is norming, meaning that for every $x\in B$ one has $\|x\|_{W_1}=\sup_{n\ge1}|\ell_n(x)|.$

A finite subset $S\subset X/A\setminus\{[A]\}$ together with a sign pattern $\sigma\colon S\to\{\pm1\}$ is called \emph{basepoint-separated} if $d_1(u,v)\ge d_1(u,[A])+d_1(v,[A])$ whenever $\sigma(u)\neq\sigma(v).$ For each basepoint-separated pair $(S,\sigma)$, define $v([A])=0$ and $v(u)=\sigma(u)\,d_1(u,[A])$ for $u\in S$, and let $f_{S,\sigma}(x):=\inf_{y\in S\cup\{[A]\}}\bigl(v(y)+d_1(x,y)\bigr) \in \mathrm{Lip}_0(X/A,d_1)$ denote the associated McShane extension~\cite{weaver2018lipschitz}. Then $\mathrm{Lip}_{d_1}(f_{S,\sigma})\le1$ and $f_{S,\sigma}(u)=\sigma(u)\,d_1(u,[A])$ for all $u\in S$. We assume that whenever a basepoint-separated pair $(S,\sigma)$ is used below, the corresponding witness $f_{S,\sigma}$ occurs among the coordinates $(\ell_n)_{n\ge1}$. Fix a strictly positive weight sequence $(w_n)_{n\ge1}\in\ell^2$. 

Define the bounded linear embedding $J\colon B\to \ell^2$ by $Jx := (w_n\,\ell_n(x))_{n\ge1},$ where $(w_n)_{n\ge1}\in\ell^2$ is a fixed strictly positive weight sequence and $(\ell_n)_{n\ge1}\subset B^\ast$ is a norming family with $\|\ell_n\|_{B^\ast}\le 1$. Fix a bounded, self-adjoint, positive, trace-class operator $\Sigma\in\mathcal L(\ell^2)$ that is diagonal with respect to the standard orthonormal basis $(e_n)_{n\ge1}$ of $\ell^2$, meaning that $\Sigma e_n=\sigma_n e_n$ for each $n\ge1$, where $\sigma_n\ge 0$ and $\sum_{n\ge1}\sigma_n<\infty$.

Let $\gamma_\Sigma$ denote the centered Gaussian Radon probability measure on $\ell^2$ with covariance operator $\Sigma$, characterized by the identity
\[
\int_{\ell^2} e^{i\langle h,u\rangle_{\ell^2}}\,d\gamma_\Sigma(u)
=
\exp\!\left(-\tfrac12\langle \Sigma h,h\rangle_{\ell^2}\right).
\]

For $u\sim\gamma_\Sigma$, define the random linear functional $\ell_u\in B^\ast$ by setting $\ell_u(x):=\langle Jx,u\rangle_{\ell^2}$ for $x\in B$. This defines a Gaussian random element of $B^\ast$ whose covariance operator $Q_{J,\Sigma}\colon B\to B^\ast$ is given by
\begin{align*}
\langle Q_{J,\Sigma}x,y\rangle_{B,B^\ast}
&=
\mathbb E\bigl[\ell_u(x)\,\ell_u(y)\bigr] \\
&=
\langle \Sigma Jx,Jy\rangle_{\ell^2},
\end{align*}
for all $x,y\in B$.

\begin{theorem}\label{thm:kJt-sigma}
Fix $t>0$ and let $u\sim\gamma_\Sigma$. Define, for $x,y\in B$,
\[
k_{J,\Sigma,t}(x,y)
:=
\mathbb E\!\left[\exp\!\bigl(i\sqrt t\,\langle J(x-y),u\rangle_{\ell^2}\bigr)\right].
\]
Then $k_{J,\Sigma,t}$ is a continuous, translation-invariant, positive definite kernel on $B$. Moreover, for all $x,y\in B$,
\[
k_{J,\Sigma,t}(x,y)
=
\exp\!\left(-\frac t2\,\|\Sigma^{1/2}J(x-y)\|_{\ell^2}^2\right).
\]
\end{theorem}

\begin{proof}
Fix $x,y\in B$. Since $J(x-y)\in\ell^2$, the characteristic function identity for the centered Gaussian measure $\gamma_\Sigma$ applied to $\sqrt t\,J(x-y)$ gives
\begin{align*}
k_{J,\Sigma,t}(x,y)
&=
\int_{\ell^2}
\exp\!\bigl(i\langle \sqrt t\,J(x-y),u\rangle_{\ell^2}\bigr)\,d\gamma_\Sigma(u) \\
&=
\exp\!\left(-\frac t2\,\langle \Sigma J(x-y),J(x-y)\rangle_{\ell^2}\right) \\
&=
\exp\!\left(-\frac t2\,\|\Sigma^{1/2}J(x-y)\|_{\ell^2}^2\right),
\end{align*}
which yields the closed form.

Translation invariance is immediate from the definition. Continuity follows from the boundedness of $\Sigma^{1/2}J\colon B\to\ell^2$ and continuity of the exponential function.

To verify positive definiteness, let $x_1,\dots,x_m\in B$ and $c_1,\dots,c_m\in\mathbb C$. By Fubini's theorem,
\[
\sum_{j,k=1}^m c_j\overline{c_k}\,k_{J,\Sigma,t}(x_j,x_k)
=
\int_{\ell^2}
\Bigl|\sum_{j=1}^m c_j
\exp\!\bigl(i\sqrt t\,\langle Jx_j,u\rangle_{\ell^2}\bigr)\Bigr|^2
\,d\gamma_\Sigma(u)
\ge 0,
\]
which completes the proof.
\end{proof}

\begin{theorem}\label{thm:lipschitz-main-sigma}
For every $f\in H_{J,\Sigma,t}$,
\[
\mathrm{Lip}_{\rho}\bigl(f|_{K(X,A)}\bigr)
\le
\sqrt t\,
\Bigl(\sum_{n\ge1}\sigma_n w_n^2\Bigr)^{1/2}\,
\|f\|_{H_{J,\Sigma,t}} .
\]
\end{theorem}

\begin{proof}
Let $\jmath\colon K(X,A)\hookrightarrow B$ denote the canonical isometric embedding. For $x,y\in B$, the RKHS feature metric satisfies
\begin{align*}
\delta_{J,\Sigma,t}(x,y)^2
&=
\|k_{J,\Sigma,t}(\cdot,x)-k_{J,\Sigma,t}(\cdot,y)\|_{H_{J,\Sigma,t}}^2 \\
&=
k_{J,\Sigma,t}(x,x)+k_{J,\Sigma,t}(y,y)-2\,\Re k_{J,\Sigma,t}(x,y).
\end{align*}
Since $k_{J,\Sigma,t}(x,x)=1$ and $k_{J,\Sigma,t}(x,y)\in(0,1]$ is real, one has $\delta_{J,\Sigma,t}(x,y)^2 = 2-2k_{J,\Sigma,t}(x,y).$ Using the closed form $k_{J,\Sigma,t}(x,y) = \exp\!\left(-\frac t2\,\|\Sigma^{1/2}J(x-y)\|_{\ell^2}^2\right)$ gives
\begin{align*}
\delta_{J,\Sigma,t}(x,y)^2
&= 2\Bigl(1-\exp\!\bigl(-\tfrac t2\|\Sigma^{1/2}J(x-y)\|_{\ell^2}^2\bigr)\Bigr) \\
&\le t\,\|\Sigma^{1/2}J(x-y)\|_{\ell^2}^2 .
\end{align*}
By the reproducing property and Cauchy--Schwarz, for $\alpha,\beta\in K(X,A)$, $|f(\alpha)-f(\beta)| \le \|f\|_{H_{J,\Sigma,t}}\, \delta_{J,\Sigma,t}\bigl(\jmath(\alpha),\jmath(\beta)\bigr).$ Using the previous estimate with $x=\jmath(\alpha)$ and $y=\jmath(\beta)$ yields $|f(\alpha)-f(\beta)| \le \sqrt t\,\|f\|_{H_{J,\Sigma,t}}\, \|\Sigma^{1/2}J(\jmath(\alpha)-\jmath(\beta))\|_{\ell^2}.$

Since $Jx=(w_n\,\ell_n(x))_{n\ge1}$ and $\Sigma e_n=\sigma_n e_n$, one has for every $x\in B$, $\|\Sigma^{1/2}Jx\|_{\ell^2}^2 = \sum_{n\ge1}\sigma_n w_n^2\,\ell_n(x)^2.$ This series converges absolutely because $\|\ell_n\|_{B^\ast}\le 1$ implies $|\ell_n(x)|\le \|x\|_{W_1}$, and hence $\sum_{n\ge1}\sigma_n w_n^2\,\ell_n(x)^2 \le \Bigl(\sum_{n\ge1}\sigma_n w_n^2\Bigr)\,\|x\|_{W_1}^2.$ Therefore, $\|\Sigma^{1/2}Jx\|_{\ell^2} \le \Bigl(\sum_{n\ge1}\sigma_n w_n^2\Bigr)^{1/2}\,\|x\|_{W_1}.$

Applying this with $x=\jmath(\alpha)-\jmath(\beta)$ and using that $\jmath$ is isometric gives $\|\Sigma^{1/2}J(\jmath(\alpha)-\jmath(\beta))\|_{\ell^2} \le \Bigl(\sum_{n\ge1}\sigma_n w_n^2\Bigr)^{1/2}\,\rho(\alpha,\beta).$ Combining the previous inequalities yields
\[
|f(\alpha)-f(\beta)|
\le
\sqrt t\,
\Bigl(\sum_{n\ge1}\sigma_n w_n^2\Bigr)^{1/2}\,
\|f\|_{H_{J,\Sigma,t}}\,
\rho(\alpha,\beta).
\]
Taking the supremum over $\alpha\neq\beta$ proves the claim.
\end{proof}

Theorem~\ref{thm:lipschitz-main-sigma} implies a continuous, non-surjective embedding $H_{J,\Sigma,t}\hookrightarrow \mathrm{Lip}_\rho\bigl(K(X,A)\bigr).$ In particular, whenever a function or vectorization can be shown to lie in $H_{J,\Sigma,t}$, its $\rho$--Lipschitz continuity follows automatically, with an explicit upper bound on the Lipschitz constant determined by the kernel parameters.

In preparation for Theorem~\ref{thm:finite-max-certificate-slack}, we fix $\delta\in(0,1)$ and $t>0$, and let $k_{J,\Sigma,t}$ be the Gaussian kernel defined earlier. Let $g=\sum_{u\in S} n_u e_u\in K(X,A)\setminus\{0\}$, where $S\subseteq X/A\setminus\{[A]\}$ is finite, $n_u\in\mathbb Z\setminus\{0\}$, and the sign pattern $\sigma_g(u)=\mathrm{sign}(n_u)$ makes $(S,\sigma_g)$ basepoint-separated. 
For $\varepsilon\in(0,1)$, define the following auxiliary quantities:
\begin{align*}
T &:= S\cup\{[A]\},\\
f_0(u) &:= \sigma_g(u)\,d_1(u,[A]) \quad (u\in S), 
\qquad f_0([A]) := 0,\\
f_\varepsilon &:= (1-\varepsilon)f_0,\\
\Delta(S,g,\varepsilon)
&:= \frac12\min_{x\neq y\in T}
\Bigl(d_1(x,y)-(1-\varepsilon)|f_0(x)-f_0(y)|\Bigr),\\
\mathcal V(S,g,\varepsilon)
&:= \Bigl\{
v:T\to \Delta(S,g,\varepsilon)\mathbb Z :
v([A])=0,\\
&\hspace{2.8em}
|v(x)-v(y)|\le d_1(x,y)\ \ \forall x,y\in T
\Bigr\},\\
L(v)
&:= \sum_{x\in T\setminus\{[A]\}}
\Bigl(2+\bigl\lfloor\log_2\bigl(1+|v(x)|/\Delta(S,g,\varepsilon)\bigr)\bigr\rfloor\Bigr).
\end{align*}

\begin{theorem}\label{thm:finite-max-certificate-slack}
For every $\varepsilon\in(0,1)$,
\begin{equation}\label{eq:finite-max-slack}
\mathcal M(g)
\le
\frac{1}{1-\varepsilon}
\left(
\sqrt{
\frac{2}{t}\,
\frac{\log\!\bigl(1/k_{J,\Sigma,t}(g,0)\bigr)}
{\displaystyle
\max_{v\in\mathcal V(S,g,\varepsilon)} 2^{-(2+3\delta)L(v)}
}
}
+
\frac{\Delta(S,g,\varepsilon)}{2}\sum_{u\in S}|n_u|
\right).
\end{equation}
\end{theorem}

\begin{proof}
By basepoint separation, $f_0$ is $1$-Lipschitz on $T$ and
\begin{align*}
\sum_{u\in S} n_u f_0(u)
&= \sum_{u\in S} |n_u|\, d_1(u,[A]) \\
&= \mathcal M(g).
\end{align*}
For distinct points $x,y\in T$, the function $f_\varepsilon=(1-\varepsilon)f_0$ satisfies
\begin{align*}
|f_\varepsilon(x)-f_\varepsilon(y)|
&\le (1-\varepsilon)\,d_1(x,y) \\
&< d_1(x,y).
\end{align*}
so $f_\varepsilon$ lies in the interior of the polytope of $1$-Lipschitz functions on $T$. By the definition of $\Delta(S,g,\varepsilon)$, any $v:T\to\mathbb R$ with $\|v-f_\varepsilon\|_{\ell^\infty(T)}\le \Delta(S,g,\varepsilon)$ is $1$-Lipschitz on $T$.

Define $v_\Delta:T\to\Delta(S,g,\varepsilon)\mathbb Z$ by rounding each value $f_\varepsilon(x)$ to the nearest multiple of $\Delta(S,g,\varepsilon)$, with $v_\Delta([A])=0$. Then $\|v_\Delta-f_\varepsilon\|_{\ell^\infty(T)}\le \tfrac12\Delta(S,g,\varepsilon),$ hence $v_\Delta\in\mathcal V(S,g,\varepsilon)$.

Moreover, $\sum_{u\in S}n_u v_\Delta(u) = \sum_{u\in S}n_u f_\varepsilon(u) + \sum_{u\in S}n_u\bigl(v_\Delta(u)-f_\varepsilon(u)\bigr),$ and therefore $\Bigl|\sum_{u\in S}n_u v_\Delta(u)\Bigr| \ge (1-\varepsilon)\mathcal M(g) - \tfrac12\Delta(S,g,\varepsilon)\sum_{u\in S}|n_u|.$

For each $v\in\mathcal V(S,g,\varepsilon)$ define $\ell_v\in B^\ast$ by the McShane extension $\ell_v(x):=\inf_{y\in T}\bigl(v(y)+d_1(x,y)\bigr).$ For every $x\in X/A$ this defines a $1$-Lipschitz function, and in particular $\|\ell_v\|_{B^\ast}\le 1$ and $\ell_v(u)=v(u)$ for all $u\in S$.

Fix a prefix-free binary encoding of each integer $m_x(v)$ and encode the vector $(m_x(v))_{x\in T\setminus\{[A]\}}$ by concatenation, yielding a prefix-free codeword $c(v)$ of length exactly $L(v)$. By Kraft's inequality, $\sum_{v\in\mathcal V(S,g,\varepsilon)} 2^{-L(v)} \le 1.$ Assign weights as
\begin{align*}
w(v) &:= 2^{-(\frac12+\delta)L(v)},\\
\sigma(v) &:= 2^{-(1+\delta)L(v)},\\
\sum_v w(v)^2
&= \sum_v 2^{-(1+2\delta)L(v)},\\
&\le \sum_v 2^{-L(v)} \le 1,\\
\sum_v \sigma(v)
&= \sum_v 2^{-(1+\delta)L(v)},\\
&\le \sum_v 2^{-L(v)} \le 1.
\end{align*}

By construction, $\sigma(v)w(v)^2 = 2^{-(2+3\delta)L(v)}.$ By the closed form of the kernel,
\begin{align*}
\log\!\bigl(1/k_{J,\Sigma,t}(g,0)\bigr)
&= \frac t2 \sum_v \sigma(v) w(v)^2\,|\ell_v(g)|^2 \\
&\ge \frac t2\, \sigma(v) w(v)^2
  \Bigl|\sum_{u\in S} n_u v(u)\Bigr|^2 .
\end{align*}
for every $v\in\mathcal V(S,g,\varepsilon)$. Applying this inequality to $v_\Delta$ and using the bound from Step~1 gives
\begin{align*}
\log\!\bigl(1/k_{J,\Sigma,t}(g,0)\bigr)
&\ge \frac t2
\left(
(1-\varepsilon)\mathcal M(g)
-
\tfrac12\Delta(S,g,\varepsilon)\sum_{u\in S}|n_u|
\right)^2 \\
&
\times \max_{v\in\mathcal V(S,g,\varepsilon)} 2^{-(2+3\delta)L(v)}.
\end{align*}
Rearranging this inequality gives~\eqref{eq:finite-max-slack}.
\end{proof}

Kernel values do not determine a unique underlying element of $K(X,A)$, even up to stability: large families of distinct virtual persistence diagrams are necessarily identified by noninjective, translation--invariant Gaussian kernels built from Hilbertian seminorms. Exact recovery from RKHS data is therefore impossible. Theorem~\ref{thm:finite-max-certificate-slack} provides a quantitative inverse bound: given a kernel value $k_{J,\Sigma,t}(g,0)$ together with finite support information, it gives an explicit upper bound on the diagrammatic mass $\mathcal M(g)$, a metric invariant defined directly in terms of $(X/A,d_1)$.

\begin{theorem}\label{thm:entropy-feature}
Fix $t>0$. For every nonempty subset $S\subseteq K(X,A)$ and every $\varepsilon>0$,
\[
N_{\delta_{J,\Sigma,t}}(S,\varepsilon)
\le
N_{\rho}\!\left(S,\frac{\varepsilon}{\sqrt t\,\|\Sigma^{1/2}J\|}\right).
\]
\end{theorem}

\begin{proof}
Fix $x,y\in K(X,A)$ and set $h:=x-y$. As in the proof of
Theorem~\ref{thm:lipschitz-main-sigma},
\begin{align*}
\delta_{J,\Sigma,t}(x,y)^2
&=
2\Bigl(1-\exp\!\bigl(-\tfrac t2\|\Sigma^{1/2}Jh\|_{\ell^2}^2\bigr)\Bigr) \\
&\le
t\,\|\Sigma^{1/2}Jh\|_{\ell^2}^2.
\end{align*}
Since $\|\Sigma^{1/2}Jh\|_{\ell^2}\le \|\Sigma^{1/2}J\|\,\|h\|_{W_1}$ and
$\|h\|_{W_1}=\rho(x,y)$ on $K(X,A)$, it follows that $\delta_{J,\Sigma,t}(x,y) \le \sqrt t\,\|\Sigma^{1/2}J\|\,\rho(x,y).$ Set $\eta:=\varepsilon/(\sqrt t\,\|\Sigma^{1/2}J\|)$. If $S$ is covered by $\rho$-balls of radius $\eta$, then the preceding inequality implies that each such $\rho$-ball is contained in a $\delta_{J,\Sigma,t}$-ball of radius $\varepsilon$. Therefore, the same centers give a $\delta_{J,\Sigma,t}$-cover of $S$ by $\varepsilon$-balls, and taking minimal cardinalities yields the stated inequality.
\end{proof}

Covering numbers with respect to the feature metric $\delta_{J,\Sigma,t}$ are controlled by those induced by the Grothendieck metric $\rho$. For every subset $S\subseteq K(X,A)$ and every $\varepsilon>0$, $N_{\delta_{J,\Sigma,t}}(S,\varepsilon) \le N_{\rho}\!\left(S,\varepsilon/(\sqrt t\,\|\Sigma^{1/2}J\|)\right)$. The same inequality implies a one-sided inverse bound when only kernel-side distances are available: for any finite $S\subset K(X,A)$ and any $\varepsilon>0$, it follows that $N_{\rho}\!\left(S,\varepsilon/(\sqrt t\, \|\Sigma^{1/2}J\|)\right) \ge N_{\delta_{J,\Sigma,t}}(S,\varepsilon)$. Thus, although the kernel embedding is noninjective and does not determine $S$ uniquely, covering numbers computed from $\delta_{J,\Sigma,t}$ provide lower bounds on the covering numbers of $S$ with respect to $\rho$.

\subsection{Comparison of covariance structures}
\label{subsec:main-comparison}

Fix $t>0$. By Theorem~\ref{thm:kJt-sigma}, every Gaussian kernel constructed in Section~\ref{subsec:main-rkhs} is completely determined by the associated Hilbertian seminorm $s_{J,\Sigma}\colon B\to[0,\infty)$, defined for $x\in B$ by $s_{J,\Sigma}(x):=\|\Sigma^{1/2}Jx\|_{\ell^2}.$ In terms of this seminorm, the kernel admits the translation-invariant representation $k_{J,\Sigma,t}(x,y) = \exp\!\left(-\frac t2\,s_{J,\Sigma}(x-y)^2\right),$ for $x,y \in B,$ and the induced RKHS feature metric satisfies the exact scalar identity
\[
\delta_{J,\Sigma,t}(x,y) = \phi_t\bigl(s_{J,\Sigma}(x-y)\bigr),
\]
\[
\phi_t(r)=\sqrt{2\left(1-\exp\!\left(-\frac t2 r^2\right)\right)},
\]
for $r \ge 0.$

Conversely, let $s\colon B\to[0,\infty)$ be any continuous Hilbertian seminorm. Then there exists a unique bounded symmetric positive operator $Q\colon B\to B^\ast$ such that $s(x)^2=\langle Qx,x\rangle_{B,B^\ast}.$ Equivalently, the bilinear form $\langle x,y\rangle_Q:=\langle Qx,y\rangle_{B,B^\ast}$ defines a continuous semi-inner product on $B$. Modding out the null space $\{x\in B:\langle Qx,x\rangle_Q=0\}$ and completing yields a Hilbert space $H_Q$, unique up to isometry, together with a bounded linear map $T_Q\colon B\to H_Q$ satisfying $Q=T_Q^\ast T_Q$ and $s(x)=\|T_Qx\|_{H_Q},$ for $x\in B$. Choosing an isometric embedding of $H_Q$ into $\ell^2$ then yields a representation of the form $Q=J^\ast\Sigma J$ for suitable operators $(J,\Sigma)$ as in Section~\ref{subsec:main-rkhs}.

In particular, the parametrization by pairs $(J,\Sigma)$ is universal: the family of kernels $k_{J,\Sigma,t}$ coincides exactly with the class of all Gaussian, translation-invariant kernels on $B$ generated by continuous Hilbertian seminorms. Different choices of $(J,\Sigma)$ give rise to the same kernel if and only if they induce the same Hilbertian seminorm on $B$.

If $T$ is injective on $B$--equivalently, if the Hilbertian seminorm $s_{J,\Sigma}$ is a norm on $B$, for example when all diagonal entries $\sigma_n>0$ on $\overline{\mathrm{ran}(J)}$--then $k_{J,\Sigma,t}$ is \emph{characteristic} on the space of Borel probability measures on $B$. Indeed, in this case, $T$ is a Borel isomorphism from $B$ onto its image, and the kernel coincides there with the standard Gaussian kernel on a separable Hilbert space, which is known to be characteristic. Conversely, for an arbitrary choice of $\Sigma$, the kernel $k_{J,\Sigma,t}$ is always characteristic on the quotient $B/\ker s_{J,\Sigma}$: it distinguishes probability measures up to collapsing precisely those directions on which $T$ vanishes. Thus $k_{J,\Sigma,t}$ is characteristic on $B$ if and only if $\ker s_{J,\Sigma}=\{0\}$.

An entirely analogous dichotomy holds for universality. If $T$ is injective on $B$ (equivalently, $s_{J,\Sigma}$ is a norm), then for every compact subset $K\subset B$ the restriction of $k_{J,\Sigma,t}$ to $K\times K$ is \emph{universal}: its reproducing kernel Hilbert space is dense in $C(K)$ with respect to the uniform norm. This follows from the universality of the Gaussian kernel on Hilbert spaces together with the fact that $T$ restricts to a homeomorphism from $K$ onto the compact set $T(K)$. In general, for an arbitrary $\Sigma$, the kernel remains universal on compact subsets of $B$ modulo the nullspace of $T$: on any compact $K\subset B$, it can approximate exactly those continuous functions that are constant along cosets of $\ker T\cap K$.

\begin{theorem}\label{thm:seminorm-bilipschitz-iff}
Let $(J_1,\Sigma_1)$ and $(J_2,\Sigma_2)$ be admissible choices as in
Section~\ref{subsec:main-rkhs}, and let
$Q_i:=Q_{J_i,\Sigma_i}\colon B\to B^\ast$ be the associated covariance operators. Define
Hilbertian seminorms $s_i(x):=\langle Q_i x,x\rangle_{B,B^\ast}^{1/2}$ for $x\in B,\ i\in\{1,2\}.$ Then
\[
\ker(Q_1)=\ker(Q_2),
\]
\[
0<
\inf_{\langle Q_1x,x\rangle>0}
\frac{\langle Q_2x,x\rangle}{\langle Q_1x,x\rangle}
\le
\sup_{\langle Q_1x,x\rangle>0}
\frac{\langle Q_2x,x\rangle}{\langle Q_1x,x\rangle}
<\infty,
\]
if and only if
\[
s_2(x)\le
\left(
\sup_{\langle Q_1y,y\rangle>0}
\frac{\langle Q_2y,y\rangle}{\langle Q_1y,y\rangle}
\right)^{1/2}
s_1(x),
\]
\[
s_2(x)\ge
\left(
\inf_{\langle Q_1y,y\rangle>0}
\frac{\langle Q_2y,y\rangle}{\langle Q_1y,y\rangle}
\right)^{1/2}
s_1(x),
\]
for all $x\in B$.
\end{theorem}

\begin{proof}
If $\{x\in B:\langle Q_1x,x\rangle_{B,B^\ast}>0\}=\varnothing$, then $Q_1=0$ and hence $\ker(Q_1)=B$. In this case, the displayed kernel identity forces $\ker(Q_2)=B$, so $Q_2=0$ and $s_1\equiv s_2\equiv 0$. Both directions of the claimed equivalence are then immediate. We therefore assume that there exists $x\in B$ with $\langle Q_1x,x\rangle_{B,B^\ast}>0$.

First assume that
\[
\ker(Q_1)=\ker(Q_2),
\]
\[
0<
\inf_{\langle Q_1x,x\rangle>0}
\frac{\langle Q_2x,x\rangle}{\langle Q_1x,x\rangle}
\le
\sup_{\langle Q_1x,x\rangle>0}
\frac{\langle Q_2x,x\rangle}{\langle Q_1x,x\rangle}
<\infty.
\]
Define
\[
\alpha:=
\inf_{\langle Q_1x,x\rangle>0}
\frac{\langle Q_2x,x\rangle}{\langle Q_1x,x\rangle},
\]
\[
\beta:=
\sup_{\langle Q_1x,x\rangle>0}
\frac{\langle Q_2x,x\rangle}{\langle Q_1x,x\rangle},
\]
so that $0<\alpha\le \beta<\infty$. Fix $x\in B$. If $\langle Q_1x,x\rangle>0$, then by the definition of $\alpha$ and $\beta$, $\alpha \le \frac{\langle Q_2x,x\rangle}{\langle Q_1x,x\rangle} \le \beta,$ and hence
\[
\alpha\,\langle Q_1x,x\rangle\le \langle Q_2x,x\rangle\le \beta\,\langle Q_1x,x\rangle.
\]
Taking square roots and using $s_i(x)^2=\langle Q_i x,x\rangle$ yields
\[
s_2(x)\ge \sqrt{\alpha}\,s_1(x),
\]
\[
s_2(x)\le \sqrt{\beta}\,s_1(x).
\]
If instead $\langle Q_1x,x\rangle=0$, then $x\in\ker(Q_1)=\ker(Q_2)$, so $\langle Q_2x,x\rangle=0$ and $s_1(x)=s_2(x)=0$, and the same inequalities hold trivially. This proves the two displayed seminorm bounds for all $x\in B$.

Conversely, assume that for all $x\in B$,
\[
s_2(x)\le
\left(
\sup_{\langle Q_1y,y\rangle>0}
\frac{\langle Q_2y,y\rangle}{\langle Q_1y,y\rangle}
\right)^{1/2}
s_1(x),
\]
\[
s_2(x)\ge
\left(
\inf_{\langle Q_1y,y\rangle>0}
\frac{\langle Q_2y,y\rangle}{\langle Q_1y,y\rangle}
\right)^{1/2}
s_1(x).
\]
Let $\alpha$ and $\beta$ be defined as above. By hypothesis, $\alpha>0$ and $\beta<\infty$. We first prove $\ker(Q_1)=\ker(Q_2)$. If $x\in\ker(Q_1)$, then $s_1(x)=0$ and the upper inequality gives $s_2(x)\le 0$, hence $s_2(x)=0$ and $x\in\ker(Q_2)$. Conversely, if $x\in\ker(Q_2)$, then $s_2(x)=0$ and the lower inequality gives $0\ge \sqrt{\alpha}\,s_1(x)$. Since $\alpha>0$ and $s_1(x)\ge 0$, it follows that $s_1(x)=0$, so $x\in\ker(Q_1)$. Thus $\ker(Q_1)=\ker(Q_2)$.

To obtain the Rayleigh bounds~\cite{10.5555/280490}, fix $x\in B$ with $\langle Q_1x,x\rangle>0$, equivalently $s_1(x)>0$. Dividing the two seminorm inequalities by $s_1(x)$ and squaring yields
\[
\alpha
\le
\frac{s_2(x)^2}{s_1(x)^2}
=
\frac{\langle Q_2x,x\rangle}{\langle Q_1x,x\rangle}
\le
\beta.
\]
Taking the infimum and supremum over $\{x:\langle Q_1x,x\rangle>0\}$ gives exactly the displayed Rayleigh bounds~\cite{10.5555/280490}. This completes the proof.
\end{proof}

\subsection{Random Fourier features}
\label{subsec:main-rff}

Fix $t>0$ and retain the canonical embedding $J\colon B\to\ell^2$ and diagonal trace--class covariance operator $\Sigma$ as in Section~\ref{subsec:main-rkhs}. Let $u_1,\dots,u_R$ be independent samples from the centered Gaussian Radon probability measure $\gamma_\Sigma$ on $\ell^2$. Define
\[
\widehat k_R(x,y)
:=
\frac1R\sum_{r=1}^R
\exp\!\bigl(i\sqrt t\,\langle J(x-y),u_r\rangle_{\ell^2}\bigr),
\]
for $x,y\in B,$ and the random feature map
\[
\Phi_R(x)
:=
\frac1{\sqrt R}
\bigl(
e^{i\sqrt t\,\langle Jx,u_1\rangle_{\ell^2}},
\dots,
e^{i\sqrt t\,\langle Jx,u_R\rangle_{\ell^2}}
\bigr)
\in\mathbb C^R .
\]

\begin{lemma}\label{lem:rff-uniform}
Let $S\subset B$ be finite and let $\hat\varepsilon>0$. Then
\[
\mathbb P\!\left(
\max_{x,y\in S}
\bigl|\widehat k_R(x,y)-k_{J,\Sigma,t}(x,y)\bigr|>\hat\varepsilon
\right)
\le
4|S|^2\exp\!\left(-\frac{R\hat\varepsilon^2}{4}\right).
\]
\end{lemma}

\begin{proof}
For fixed $x,y\in B$, the random variables $\exp(i\sqrt t\,\langle J(x-y),u_r\rangle_{\ell^2})$ have real and imaginary parts in $[-1,1]$ and expectation $k_{J,\Sigma,t}(x,y)$. Hoeffding's inequality yields $\mathbb P\bigl(|\widehat k_R(x,y)-k_{J,\Sigma,t}(x,y)|>\hat\varepsilon\bigr) \le 4\exp\!\left(-\frac{R\hat\varepsilon^2}{4}\right).$ A union bound over all $(x,y)\in S\times S$ gives the claim.
\end{proof}

\begin{corollary}\label{cor:rff-lipschitz}
The random feature map $\Phi_R\colon (B,\|\cdot\|_{W_1})\longrightarrow(\mathbb C^R,\|\cdot\|_{\ell^2})$ is Lipschitz, with Lipschitz constant satisfying
\[
\mathrm{Lip}(\Phi_R)
\le
\sqrt t\,\|J\|\,
\left(
\frac1R\sum_{r=1}^R\|u_r\|_{\ell^2}^2
\right)^{1/2}.
\]
\end{corollary}

\begin{proof}
Fix $x,y\in B$. For real $a,b$ one has
$|e^{ia}-e^{ib}|=2|\sin((a-b)/2)|\le |a-b|$, and hence
\begin{align*}
\|\Phi_R(x)-\Phi_R(y)\|_{\ell^2}^2
&=
\frac1R\sum_{r=1}^R
\bigl|e^{i\sqrt t\,\langle Jx,u_r\rangle_{\ell^2}}
   -e^{i\sqrt t\,\langle Jy,u_r\rangle_{\ell^2}}\bigr|^2 \\
&\le
\frac tR\sum_{r=1}^R
\bigl|\langle J(x-y),u_r\rangle_{\ell^2}\bigr|^2 \\
&=
\frac tR\sum_{r=1}^R
\bigl|\langle x-y,J^\ast u_r\rangle_{B,B^\ast}\bigr|^2 \\
&\le
\frac tR\sum_{r=1}^R
\|J^\ast u_r\|_{B^\ast}^2\,\|x-y\|_{W_1}^2 .
\end{align*}
Since $\|J^\ast u_r\|_{B^\ast}\le \|J\|\,\|u_r\|_{\ell^2}$ for each $r$, the stated bound
follows by taking square roots.
\end{proof}

\begin{corollary}\label{cor:rff-mass}
Assume \emph{(H1)} and retain the notation of Theorem~\ref{thm:finite-max-certificate-slack}. Fix $\hat\varepsilon\in(0,1)$ and let $S_0\subset K(X,A)$ be finite. Then, with probability at least
\[
1-4\bigl(|S_0|+1\bigr)^2\exp\!\left(-\frac{R\hat\varepsilon^2}{4}\right),
\]
for every $g\in S_0$ satisfying the hypotheses of Theorem~\ref{thm:finite-max-certificate-slack} and such that $\Re\,\widehat k_R(g,0)\ge 2\hat\varepsilon$, one has
\[
\mathcal M(g)
\le
\frac{1}{1-\varepsilon}
\left(
\sqrt{
\frac{2}{t}\,
\frac{\log\!\bigl(1/(\Re\,\widehat k_R(g,0)-\hat\varepsilon)\bigr)}
{\displaystyle
\max_{v\in\mathcal V(S,g,\varepsilon)} 2^{-(2+3\delta)L(v)}
}
}
+
\frac{\Delta(S,g,\varepsilon)}{2}\sum_{u\in S}|n_u|
\right).
\]
\end{corollary}

\begin{proof}
Apply Lemma~\ref{lem:rff-uniform} with $S:=S_0\cup\{0\}$. On the resulting event, for every $g\in S_0$ one has $\bigl|\widehat k_R(g,0)-k_{J,\Sigma,t}(g,0)\bigr|\le \hat\varepsilon,$ and hence $\Re\,\widehat k_R(g,0)\ge k_{J,\Sigma,t}(g,0)-\hat\varepsilon.$ Equivalently, $k_{J,\Sigma,t}(g,0)\ge \Re\,\widehat k_R(g,0)-\hat\varepsilon.$ If $\Re\,\widehat k_R(g,0)\ge 2\hat\varepsilon$, then $\Re\,\widehat k_R(g,0)-\hat\varepsilon>0$ and therefore $\log\!\bigl(1/k_{J,\Sigma,t}(g,0)\bigr) \le \log\!\bigl(1/(\Re\,\widehat k_R(g,0)-\hat\varepsilon)\bigr).$ Substituting this bound into Theorem~\ref{thm:finite-max-certificate-slack} yields the claimed inequality.
\end{proof}

\begin{corollary}\label{cor:rff-entropy}
Let $S\subseteq K(X,A)$ be nonempty and let $\varepsilon>0$. If $\mathrm{Lip}(\Phi_R)>0$, then
\[
N_{\ell^2}\bigl(\Phi_R(S),\varepsilon\bigr)
\le
N_{\rho}\!\left(S,\frac{\varepsilon}{\mathrm{Lip}(\Phi_R)}\right).
\]
If $\mathrm{Lip}(\Phi_R)=0$, then $\Phi_R$ is constant and $N_{\ell^2}(\Phi_R(S),\varepsilon)=1$.
\end{corollary}

\begin{proof}
If $\mathrm{Lip}(\Phi_R)=0$, then $\Phi_R$ is constant, so $\Phi_R(S)$ is a singleton and $N_{\ell^2}(\Phi_R(S),\varepsilon)=1$. Assume $\mathrm{Lip}(\Phi_R)>0$ and set $\eta:=\varepsilon/\mathrm{Lip}(\Phi_R)$. Since $\Phi_R$ is $\mathrm{Lip}(\Phi_R)$--Lipschitz from $(K(X,A),\rho)$ into $(\mathbb C^R,\|\cdot\|_{\ell^2})$, for every $x\in K(X,A)$ one has $\Phi_R\bigl(B_\rho(x,\eta)\bigr)\subseteq B_{\ell^2}\bigl(\Phi_R(x),\varepsilon\bigr).$ Therefore any cover of $S$ by $\rho$--balls of radius $\eta$ pushes forward to a cover of $\Phi_R(S)$ by $\ell^2$--balls of radius $\varepsilon$. Taking minimal cardinalities yields the inequality.
\end{proof}

\section{Examples}
\label{sec:example}

\begin{figure}[t]
 \centering
 \begin{tabular}{cc}
  \includegraphics[width=0.45\textwidth]{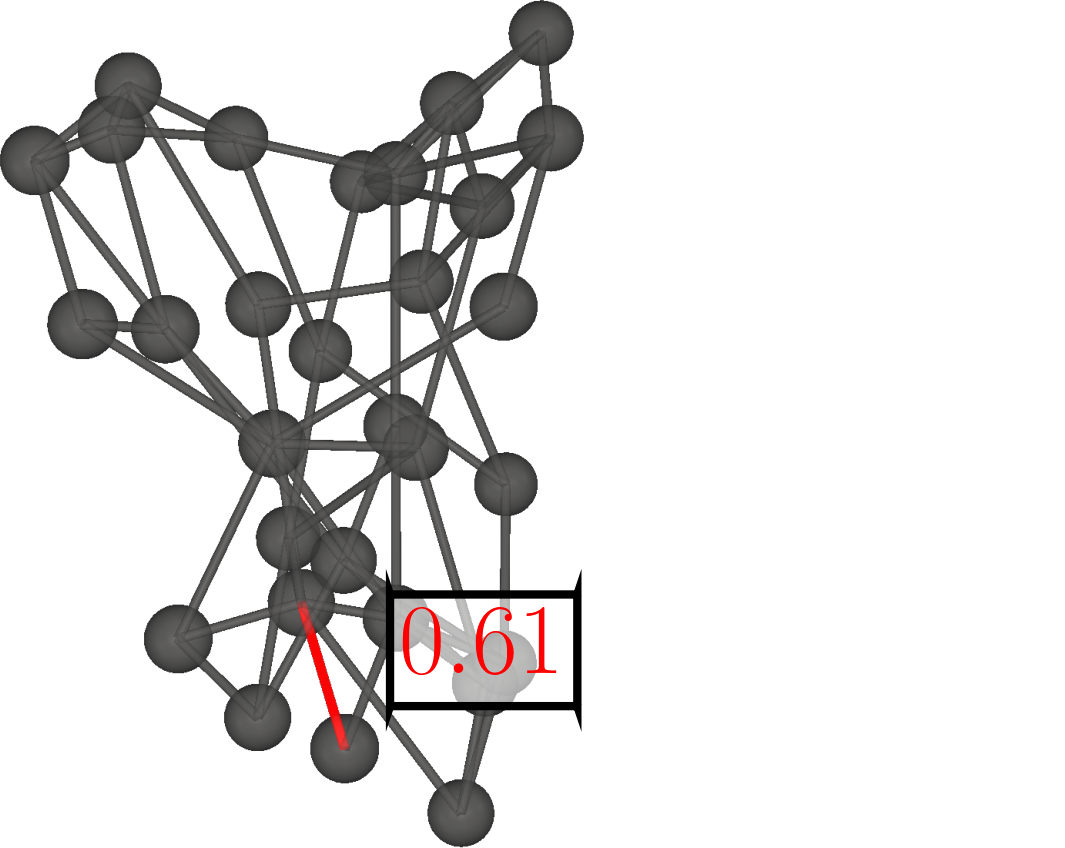} &
  \includegraphics[width=0.45\textwidth]{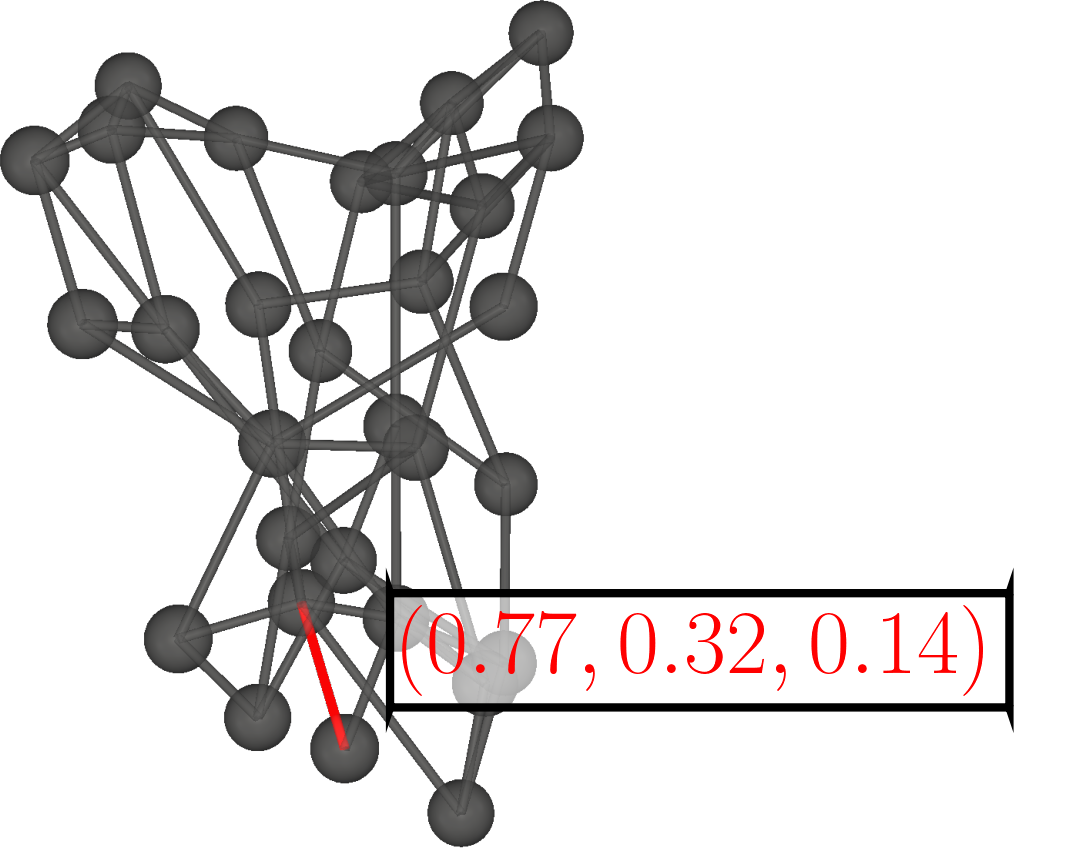} \\
  (a) $\mathbb R$ &
  (b) $\mathbb R^3$ \\
  \includegraphics[width=0.45\textwidth]{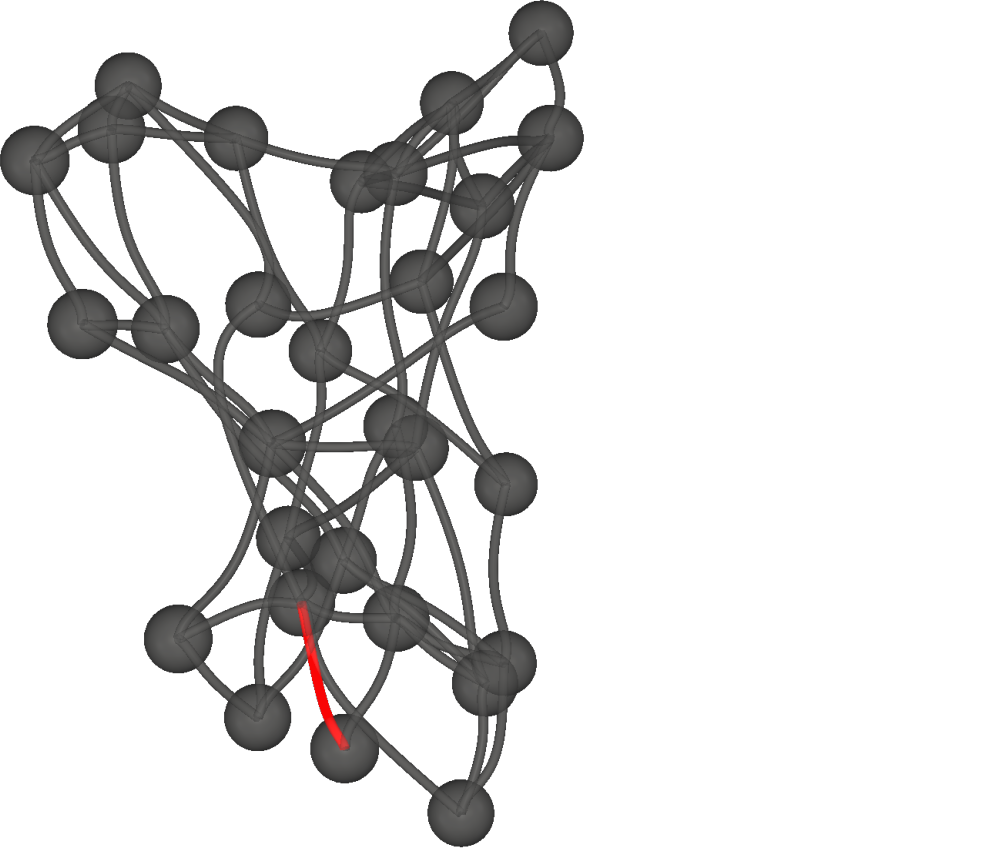} &
  \includegraphics[width=0.45\textwidth]{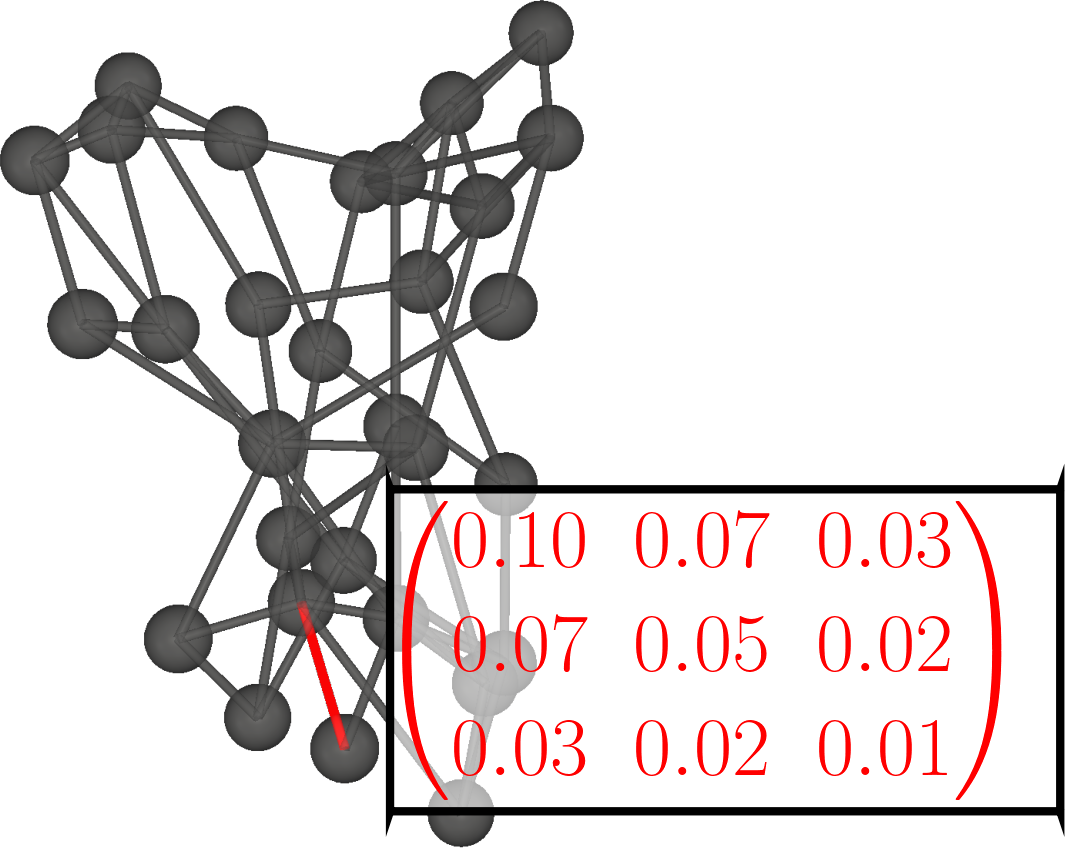} \\
  (c) $C([0,1])$ &
  (d) $\mathcal S^+_3$
 \end{tabular}
 \caption{A fixed Watts--Strogatz graph equipped with four deterministic, basis--invariant spectral edge labelings into non--uniformly discrete spaces. Each panel shows the same graph, with one representative edge highlighted; panels (a)--(d) correspond to different choices of label space.}
 \label{fig:rff-network-labels}
\end{figure}

Figure~\ref{fig:rff-network-labels} shows four deterministic spectral edge labelings on a common underlying graph. The graph is a Watts--Strogatz small--world graph with parameters $n=30$, $k=4$, $p=0.3$. In panel~(a), vertices carry the diagonal of the Laplacian pseudoinverse and each edge $\{u,v\}$ is labeled by the maximum of the endpoint values; $\mathbb R$ is equipped with its usual order and metric. In panel~(b), vertices are embedded in $\mathbb R^3$ via heat kernel diagonals at three fixed spectral time scales and edges are labeled by the coordinatewise maximum of the endpoint embeddings; the product order and Euclidean metric are used on $\mathbb R^3$. In panel~(c), vertex heat profiles relative to a fixed base vertex are differenced along edges and normalized to increasing homeomorphisms of $[0,1]$; $C([0,1])$ is equipped with the pointwise order and the supremum norm. In panel~(d), the embedding from panel~(b) is reused to assign each edge the rank--one positive semidefinite matrix given by the outer product of the difference of the endpoint vectors; the L\"owner order and Frobenius norm are used on $\mathcal S^+_3$. In each case, we use the lower--star clique filtration induced by the specified partial order, and we measure virtual persistence diagrams using the corresponding product metric on the quotient birth--death space.

More generally, let $G=(V,E)$ be a finite simple graph and let $(P,d_P,\preceq)$ be a partially ordered metric space that is not uniformly discrete. Any deterministic edge labeling $\ell\colon E\to P$ determines a $P$-indexed lower--star filtration of clique complexes by $K_p := \bigl\{\sigma\subseteq V : \ell(e)\preceq p \text{ for all edges } e\subset\sigma\bigr\},$ for $p\in P.$ Persistent homology of this filtration yields birth--death pairs in $P^2$, which define a virtual persistence diagram valued in the metric pair $(P^2,d_P\oplus d_P,A)$, where $A=\{(p,p):p\in P\}$ is the diagonal.

\begin{figure}[t]
 \centering
 \begin{tabular}{cc}
  \includegraphics[width=0.5\textwidth]{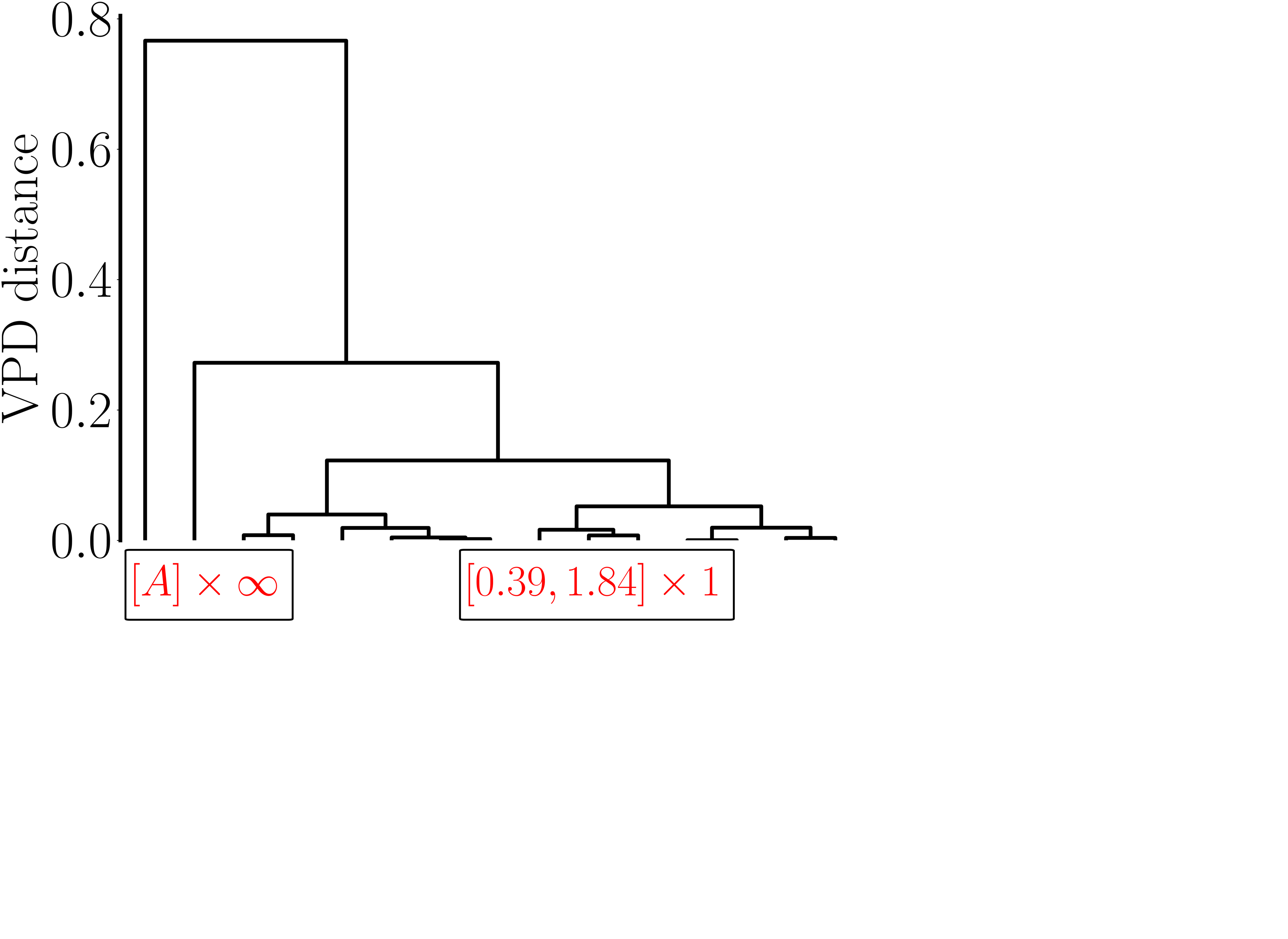} &
  \includegraphics[width=0.5\textwidth]{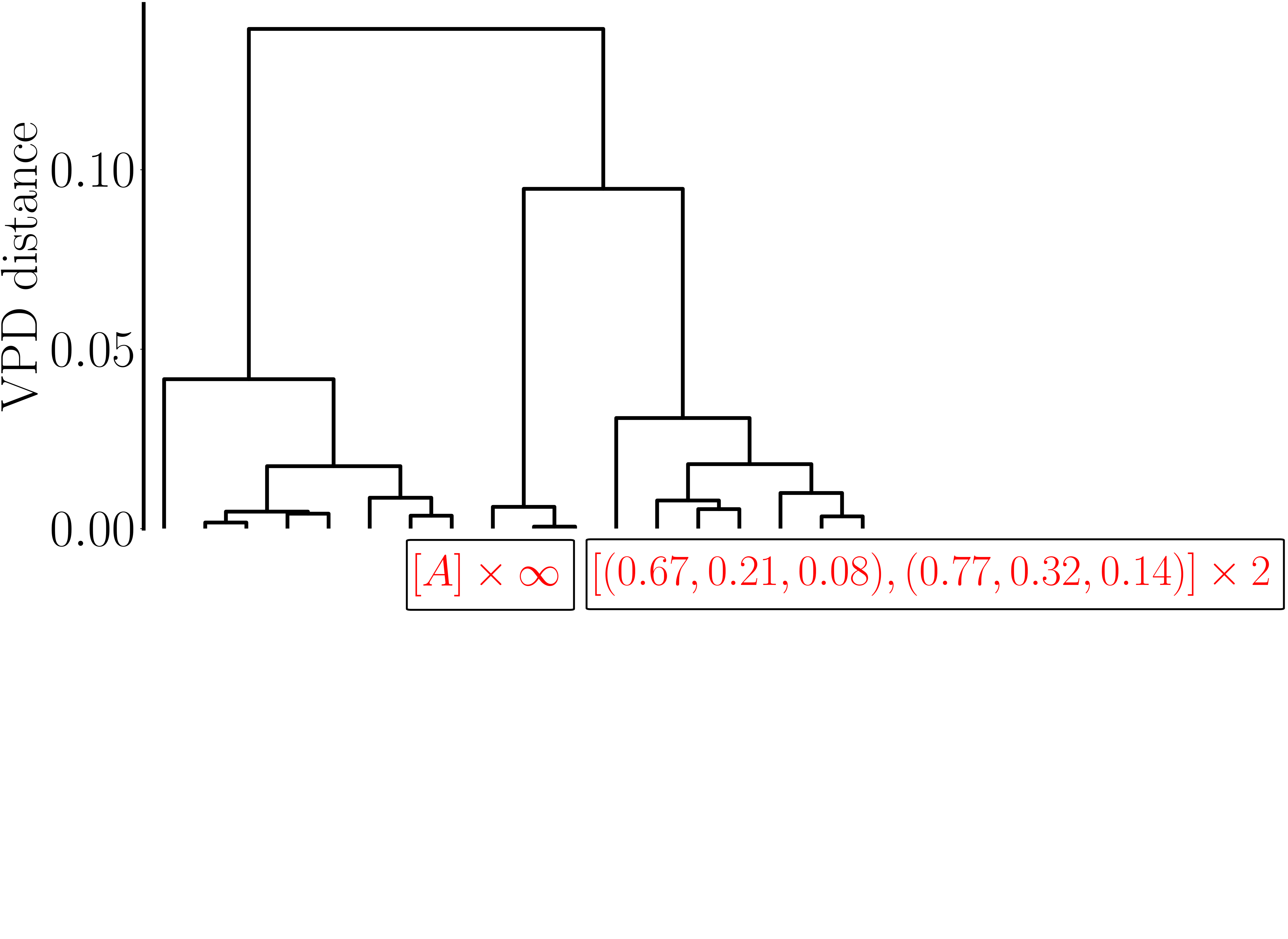} \\
  (a) $\mathbb R^2/A$ &
  (b) $(\mathbb R^3)^2/A$ \\
  \includegraphics[width=0.5\textwidth]{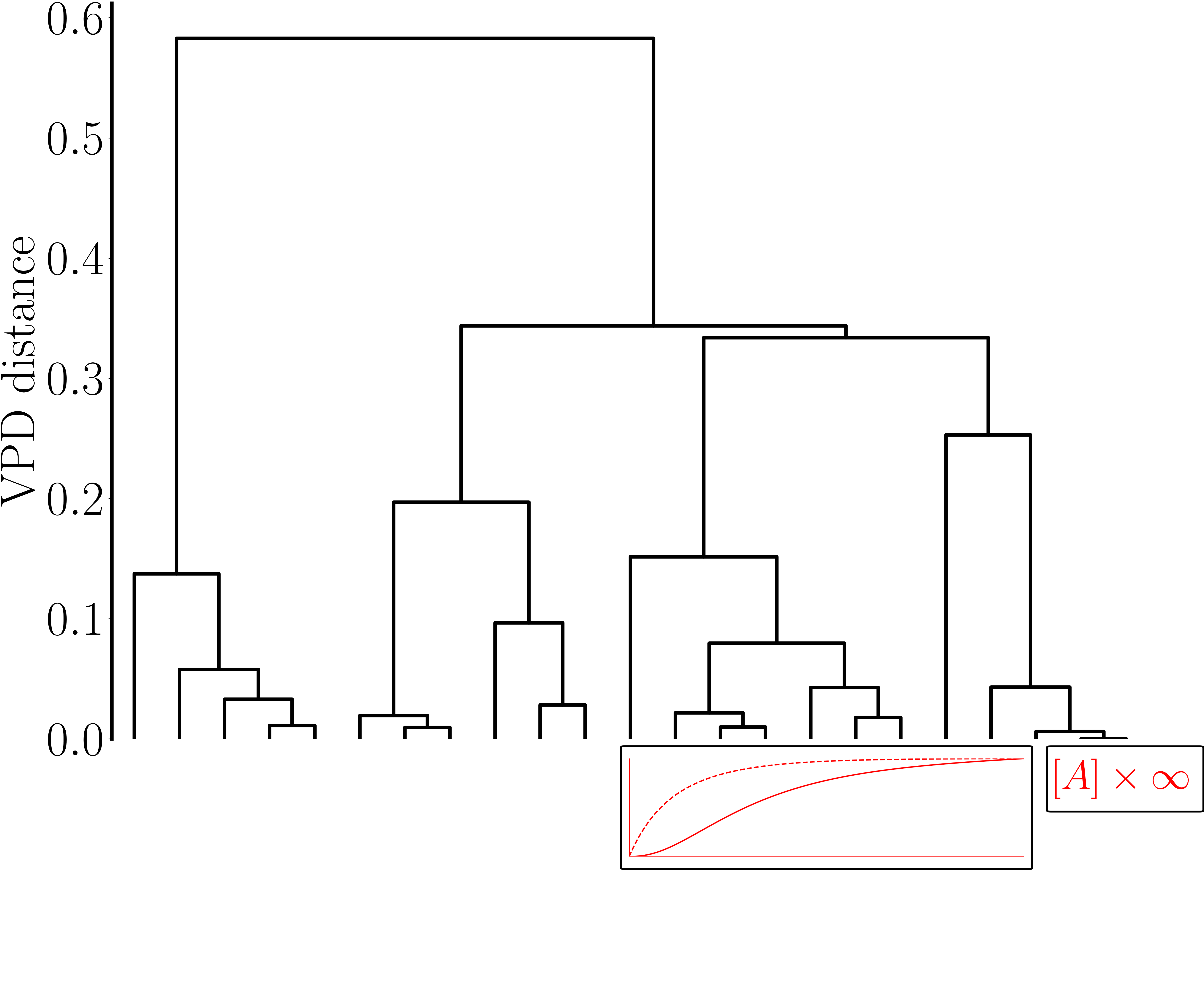} &
  \includegraphics[width=0.5\textwidth]{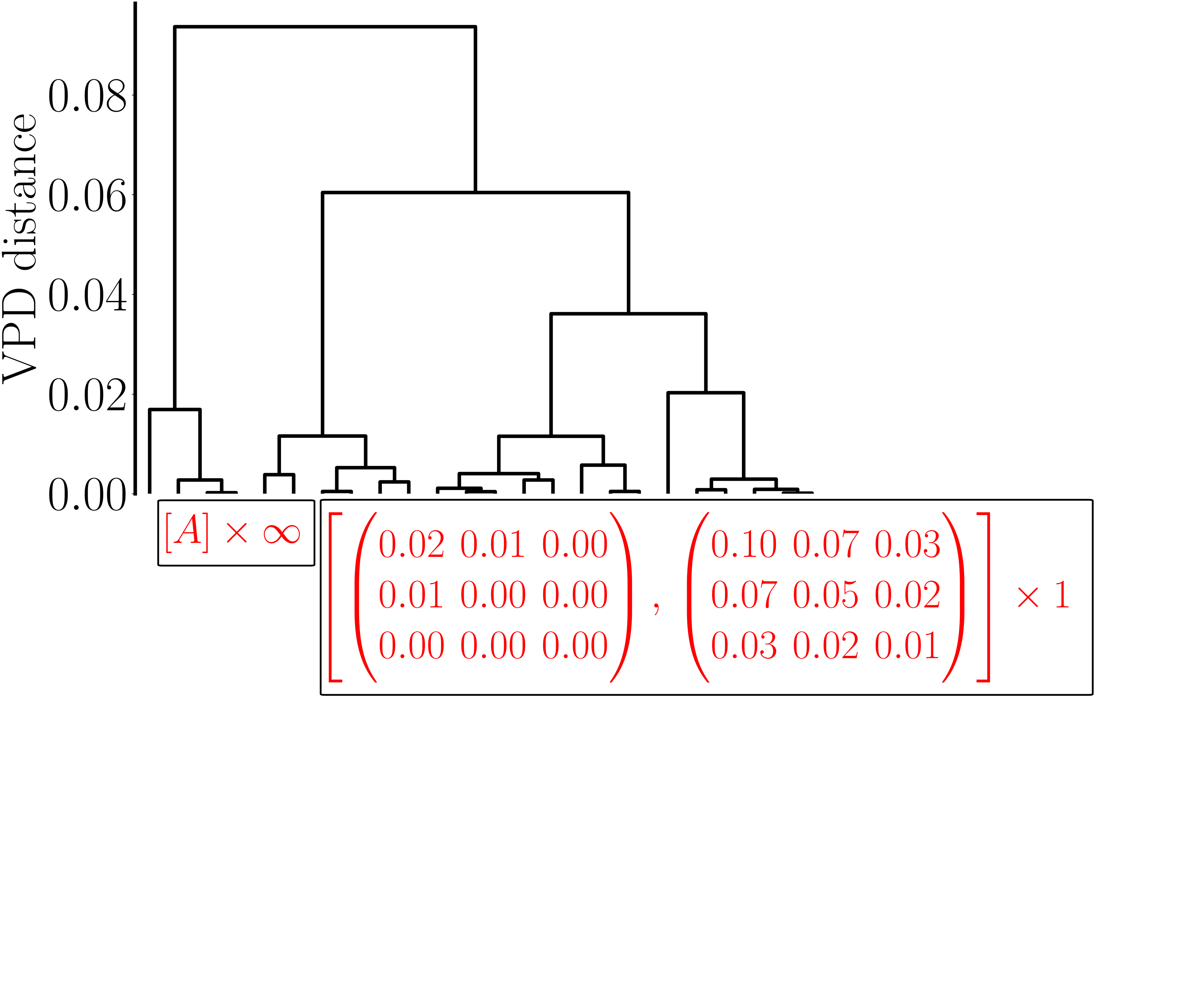} \\
  (c) $C([0,1])^2/A$ &
  (d) $(\mathcal S^+_3)^2/A$
 \end{tabular}
 \caption{Ultrametric dendrograms for the $H_1$ virtual persistence diagrams associated with the edge--labeled graphs in Figure~\ref{fig:rff-network-labels}.}
 \label{fig:rff-network-vpd}
\end{figure}

Figure~\ref{fig:rff-network-vpd} displays the single--linkage dendrogram of the multiset of birth--death classes obtained from the $H_1$ persistent homology computation in this example, viewed as a subset of $P^2/A\cup\{[A]\times\infty\}$ and clustered with respect to the $1$--strengthened metric. The dendrogram carries the subdominant ultrametric dominated by $d_1$, with branch heights equal to merge distances. Leaves are in bijection with the computed persistence classes, including the diagonal class $[A]\times\infty$, and are labeled by their birth--death data and multiplicities. This ultrametric induces a canonical filtration on the diagram that remains well defined for non--Euclidean and infinite--dimensional label spaces.

For the four edge labelings in Figures~\ref{fig:rff-network-labels} and~\ref{fig:rff-network-vpd}, we evaluate the theoretical bounds of Section~\ref{subsec:main-rff} numerically using the fixed choices $\dim H=128$, $t=1.0$, $R=100$, and $\delta=0.05$. The Lipschitz bounds for the restriction of RKHS functions to $K(X,A)$ take the values $51.32$, $44.25$, $47.72$, and $44.86$. Evaluating the random Fourier feature concentration bound at $(R,\delta)$ gives approximation bounds equal to $0.4275$, $0.4279$, $0.4279$, and $0.4283$. At a perturbation level equal to $0.1$ times the average persistence lifetime, the robustness bound evaluates to $7.47$, $0.83$, $2.23$, and $0.53$. The underlying graph, the lower--star clique filtration construction, and the kernel construction are the same in all four cases. These values isolate the contribution of the label space and its induced metric. The Lipschitz and feature approximation bounds vary little across the four labelings. The robustness bounds differ by more than an order of magnitude.

A metric space $(Y,\rho)$ arises, up to isometry, as the ambient birth--death space of a virtual persistence diagram produced by an interval-based construction if and only if $(Y,\rho)$ is isometric to $(X\times X,d\oplus d)$ for some metric space $(X,d)$, where
\[
(d\oplus d)\bigl((x_1,x_2),(y_1,y_2)\bigr):=d(x_1,y_1)+d(x_2,y_2)
\]
denotes the $\ell_1$ product metric--if and only if there exist pseudometrics $\rho_1,\rho_2$ on $Y$ such that $\rho=\rho_1+\rho_2$, each $\rho_i$ descends to a genuine metric $d_i$ on the quotient space $X_i:=Y/\!\sim_i$ (where $y\sim_i y'$ if and only if $\rho_i(y,y')=0$), and the induced equivalence relations satisfy the grid condition that every $\sim_1$--equivalence class meets every $\sim_2$--equivalence class in exactly one point. In this case, the canonical map $\Psi\colon Y\longrightarrow X_1\times X_2,$
\[
\Psi(y):=([y]_1,[y]_2),
\]
is an isometry from $(Y,\rho)$ onto the $\ell_1$ product $(X_1\times X_2,d_1\oplus d_2)$, and $(Y,\rho)$ is an $\ell_1$--square--i.e., isometric to $(X\times X,d\oplus d)$ for some metric space $(X,d)$--if and only if the factor metric spaces $(X_1,d_1)$ and $(X_2,d_2)$ are isometric. Consequently, the interval-based VPD algorithm realizes exactly those ambient metric geometries admitting such an additive two-pseudometric decomposition with transversal kernels, and no others.

The constructions shown in Figure~\ref{fig:rff-network-vpd} take values in the commutative monoid $D(X,A)$ of virtual persistence diagrams with nonnegative multiplicities. The passage to signed virtual persistence diagrams is canonical: given $\alpha,\beta\in D(X,A)$, their formal difference $\alpha-\beta$ defines an element of the Grothendieck group $K(X,A)$.

\section{Conclusion}
\label{sec:conclusion}

The central structural result of this paper is a complete classification of local compactness for the Grothendieck metric group \((K(X,A),\rho)\). Theorem~\ref{thm:classification} shows that \((K(X,A),\rho)\) is locally compact if and only if its metric topology is discrete, and Corollary~\ref{cor:input-output-discrete} identifies this condition with uniform discreteness of the pointed metric space \((X/A,d_1,[A])\). In the uniformly discrete case, \(K(X,A)\cong\bigoplus_{x\in X\setminus A}\mathbb Z\) is a discrete locally compact abelian group, and the Fourier and RKHS constructions based on Pontryagin duality and Haar measure apply as recalled in Section~\ref{subsec:background-lc-rkhs}. Outside this case, which is the generic situation under the standing hypotheses~(H1)--(H2) in Definition~\ref{def:standing-hypotheses}, \((K(X,A),\rho)\) is non-locally compact and Bochner's theorem is unavailable. To treat this non-discrete case, we pass to the Banach completion \(B=\widehat V(X,A)\), which is canonically isometric to the Lipschitz-free space \(\mathcal F(X/A,d_1)\) and infinite dimensional by Lemmas~\ref{lem:infinite-rank} and~\ref{lem:B-separable}. From a norming family \((\ell_n)\subset B^\ast\), weights \((w_n)\in\ell^2\), and a diagonal trace-class covariance operator \(\Sigma\), Section~\ref{subsec:main-rkhs} constructs a bounded linear map \(J\colon B\to\ell^2\) and an associated centered Gaussian measure with covariance operator \(Q_{J,\Sigma}\). Theorem~\ref{thm:kJt-sigma} shows that this data determines, for each \(t>0\), a translation-invariant positive definite kernel $k_{J,\Sigma,t}(x,y) = \exp\!\left(-\frac t2\,\|\Sigma^{1/2}J(x-y)\|_{\ell^2}^2\right).$ Section~\ref{subsec:main-comparison} shows that every continuous Hilbertian seminorm on \(B\) is realized by such covariance data, and Theorem~\ref{thm:seminorm-bilipschitz-iff} gives an exact criterion, in terms of Rayleigh quotients~\cite{10.5555/280490} of the associated covariance operators, for when two choices induce uniformly equivalent Gaussian kernels. When these kernels are restricted to \(K(X,A)\), Theorem~\ref{thm:lipschitz-main-sigma} provides an explicit global bound on the \(\rho\)-Lipschitz constants of all functions in the corresponding reproducing kernel Hilbert spaces, while Theorem~\ref{thm:entropy-feature} bounds the covering numbers of subsets of \(K(X,A)\) in the induced feature metric by those of the Grothendieck metric. In the opposite direction, Theorem~\ref{thm:finite-max-certificate-slack} shows that a single kernel value \(k_{J,\Sigma,t}(g,0)\), together with finite support information, yields an explicit upper bound on the diagrammatic mass \(\mathcal M(g)\). Section~\ref{subsec:main-rff} then shows that all of these bounds persist, up to explicit probabilistic error terms, when kernel evaluations are replaced by random Fourier feature approximations.

The constructions in Section~\ref{sec:main} define filtrations indexed by partially ordered metric label spaces and produce virtual persistence diagrams in the associated metric birth--death spaces. In the example of Section~5, a Watts--Strogatz graph is equipped with edge labelings into \(\mathbb R\), \(\mathbb R^3\), \(C([0,1])\), and \(\mathcal S^+_3\), each carrying a natural partial order and metric (Figure~\ref{fig:rff-network-labels}); the induced lower--star clique filtrations yield birth--death pairs in \(P^2\) and hence virtual persistence diagrams in \((P^2,d_P\oplus d_P,A)\) with diagonal \(A\) (Figure~\ref{fig:rff-network-vpd}). Although these label spaces are non-uniformly discrete and, in two cases, infinite dimensional, collapsing the diagonal places all resulting diagrams in a single Grothendieck metric group \((K(X,A),\rho)\) and its Banach completion \(B=\widehat V(X,A)\) by Theorems~\ref{thm:rho} and~\ref{thm:BE-chain}. The diagrams obtained in this way have nonnegative multiplicities and therefore lie in the submonoid \(D(X,A)\). Finite linear combinations of these diagrams are formed by adding and subtracting multiplicities coordinatewise, with distances evaluated using the translation invariance of \(\rho\). Virtual persistence diagrams are therefore treated as elements of a metric abelian group. The kernels, Gaussian measures, Banach-space linearization, and random feature maps developed in Sections~\ref{sec:main} and~\ref{subsec:main-rkhs} act on signed combinations of diagrams and extend persistence-based constructions to partially ordered metric label spaces under the Wasserstein--\(1\) geometry.

The constructions described above have several limitations when compared with the discrete, locally compact case in Section~\ref{subsec:background-lc-rkhs}. In the non-uniformly discrete case covered by~(H1)--(H2), the Gaussian kernels \(k_{J,\Sigma,t}\) depend on auxiliary choices of a norming sequence \((\ell_n)\), weights \((w_n)\), the embedding \(J\colon B\to\ell^2\), and the covariance operator \(\Sigma\); Theorem~\ref{thm:seminorm-bilipschitz-iff} compares different choices at the level of covariance operators but does not single out a canonical one. This contrasts with the finite case, where the identification \(K(X,A)\cong\mathbb Z^{|X\setminus A|}\) and the Pontryagin dual construction in Section~\ref{subsec:background-lc-rkhs} select a distinguished dual group and Haar measure without additional data. The Lipschitz bounds in Theorem~\ref{thm:lipschitz-main-sigma}, the mass bounds in Theorem~\ref{thm:finite-max-certificate-slack}, and the concentration inequalities in Lemma~\ref{lem:rff-uniform} and Corollary~\ref{cor:rff-mass} are expressed in terms of global quantities such as \(\sum_n\sigma_n w_n^2\) and worst-case combinatorial encodings, and the numerical values in Section~5 show that these estimates are often conservative. On the algorithmic side, the example based on lower-star clique filtrations along edge labelings produces virtual persistence diagrams in product spaces of the form \((P^2,d_P\oplus d_P,A)\), and the characterization at the end of Section~5 shows that this covers exactly those birth--death spaces that are isometric to \(\ell_1\)-squares. Consequently, that construction does not produce virtual persistence diagrams for all metric pairs \((X,d,A)\); metric spaces that do not admit such an \(\ell_1\)-decomposition are not reached by the example. In addition, the algorithm fixes the diagonal subset \(A\) as the trivial diagonal in a product space and collapses it to a single basepoint, whereas the general definition of virtual persistence diagrams in Section~\ref{subsec:background-vpd} allows arbitrary choices of \(A\subset X\).

These limitations suggest several directions for further work. On the analytic side, one might consider bounds such as those in Theorems~\ref{thm:lipschitz-main-sigma}, \ref{thm:finite-max-certificate-slack}, and \ref{thm:entropy-feature} that depend on finer invariants of the covariance operator \(Q_{J,\Sigma}\) and of the metric structure of \((X/A,d_1)\). The comparison principle in Theorem~\ref{thm:seminorm-bilipschitz-iff} suggests optimizing over admissible covariance operators within a fixed class in order to control Lipschitz constants and mass bounds subject to geometric constraints imposed by a given application. On the algorithmic side, it would be natural to design constructions that associate virtual persistence diagrams to more general metric pairs \((X,d,A)\), not only to product spaces with the \(\ell_1\)-metric, and to understand to what extent arbitrary metric birth--death spaces can be approximated by such constructions. A related problem is to treat nontrivial diagonal choices: given \((X,d)\), one would like methods that select subsets \(A\subset X\) that play the role of a diagonal and interact well with the Grothendieck metric \(\rho\) and with the kernels \(k_{J,\Sigma,t}\).

\section*{Statements and Declarations}

\begin{itemize}

\item \textbf{Competing Interests} The authors declare that they have no competing interests.

\item \textbf{Funding} This research received no external funding.

\item \textbf{Data Availability} Not applicable.

\item \textbf{Code Availability} The implementation used in this work is available at \url{https://github.com/cfanning8/Reproducing_Kernel_Hilbert_Spaces_for_Non_Discrete_Virtual_Persistence_Diagrams}.

\item \textbf{Authors' Contributions}
C.F. developed the theoretical framework, conducted and analyzed the examples, and wrote the manuscript. M.E.A. advised the project and provided feedback on the framework and manuscript.

\end{itemize}



\end{document}